\RequirePackage{fix-cm}
\documentclass[authoryear,smallextended]{svjour3}       
\smartqed  
\usepackage{graphicx}

\usepackage{amsmath}
\usepackage{amsfonts}
\usepackage{amssymb}
\usepackage{verbatim}
\usepackage{enumerate}
\usepackage[colorlinks,citecolor=blue,urlcolor=blue]{hyperref}
\usepackage[justification=centering]{caption}
\usepackage{booktabs,caption}
\usepackage[flushleft]{threeparttable}

\newenvironment{proofsect}[1]{\vskip0.1cm\noindent{\rmfamily\itshape #1.}}{\qed\vspace{0.15cm}}

\spnewtheorem*{Main Theorem}{Main Theorem}{\normalfont\bfseries}{\itshape}
\spnewtheorem{mylemma}[theorem]{Lemma}{\bfseries}{\itshape} 
\spnewtheorem{myproposition}[theorem]{Proposition}{\bfseries}{\itshape} 
\spnewtheorem{mycorollary}[theorem]{Corollary}{\bfseries}{\itshape} 
\spnewtheorem{mydefinition}[theorem]{Definition}{\bfseries}{\itshape} 
\spnewtheorem{myquestion}{Question}{\bfseries}{\itshape} 
\spnewtheorem{myconjecture}[myquestion]{Conjecture}{\bfseries}{\itshape} 
\numberwithin{equation}{section} \numberwithin{theorem}{section}
\usepackage[utf8]{inputenc}
\usepackage{fourier} 
\usepackage{array}
\usepackage{makecell}
\usepackage{bm}

\makeatletter
\newcommand{\vast}{\bBigg@{2.5}}
\newcommand{\Vast}{\bBigg@{5}}
\makeatother

\newcommand\setItemnumber[1]{\setcounter{enumi}{\numexpr#1-1\relax}}

\newcommand\Item[1][i]{%
	\ifx\relax#1\relax  \item \else \item[#1] \fi
	\abovedisplayskip=0pt\abovedisplayshortskip=0pt~\vspace*{-\baselineskip}}

\usepackage{amsthm}

\makeatletter
\g@addto@macro{\definition}{\itshape}
\makeatother

\usepackage{geometry}
\usepackage{natbib} 
\usepackage{filecontents}

\begin{document}

\title{Big Data Information Reconstruction on an Infinite Tree for a $4\times 4$-state Asymmetric Model with Community Effects
}

\titlerunning{Big Data Information Reconstruction on an Infinite Tree}        

\author{Wenjian Liu         \and
        Ning Ning 
}


\institute{ 	Wenjian Liu \at
	Dept.of Mathematics and Computer Science,
	Queensborough Community College, City University of New York\\
	\email{wjliu@qcc.cuny.edu}             
	\and 
	Ning Ning (Corresponding Author)\at
	Dept. of Applied Mathematics, University
	of Washington, Seattle\\
	\email{ningnin@uw.edu}\\
}

\date{}

\maketitle

\begin{abstract}
The information reconstruction problem on an infinite tree, is to collect and analyze massive data samples at the $n$th level of the tree to identify whether there is non-vanishing information of the root, as $n$ goes to infinity. This problem has wide applications in various fields such as biology, information theory and statistical physics, and its close connections to cluster learning,   data mining and deep learning have been well established in recent years. Although it has been studied in numerous contexts, the existing literatures with rigorous reconstruction thresholds established are very limited. In this paper, motivated by a classical deoxyribonucleic acid (DNA) evolution model, the F$81$ model, and taking into consideration of the Chargaff's parity rule by allowing the existence of a guanine-cytosine content bias, we study the noise channel in terms of a $4\times 4$-state asymmetric probability transition matrix with community effects, for four nucleobases of DNA. The corresponding information reconstruction problem in molecular phylogenetics is explored, by means of refined analyses of moment recursion, in-depth concentration estimates, and thorough investigations on an asymptotic $4$-dimensional nonlinear second order dynamical system. We rigorously show that the reconstruction bound is not tight when the sum of the base frequencies of adenine and thymine falls in the interval $\left(0,1/2-\sqrt{3}/6\right)\bigcup \left(1/2+\sqrt{3}/6,1\right)$, which is the first rigorous result on asymmetric noisy channels with community effects.
\keywords{Kesten-Stigum reconstruction bound \and Markov random fields on trees \and 
	Distributional recursion \and Nonlinear dynamical system}
\subclass{ 60K35 \and 82B26 \and 82B20}
\end{abstract}

\section{Introduction}
\label{intro}

\subsection{Big data and the reconstruction problem}

``Big Data", like the name implies, refers to massive data sets. 
 Finding information about the source of massive data as time evolves, is one of the toughest big data challenges. In this paper, we consider the following broadcasting process that can be considered as signals transmitting on an infinite communication tree network, as a model for propagation of a genetic property, or as a tree-indexed Markov chain.
It has two building blocks: the first is an irreducible aperiodic Markov chain on a finite characters set $\mathcal{C}$;
the second is a $d$-ary tree, which is a rooted tree with every vertex having
exactly $d$ offspring, denoted as  $\mathbb{T}=(\mathbb{V}, \mathbb{E}, \rho)$ with nodes
$\mathbb{V}$, edges $\mathbb{E}$, and root $\rho\in \mathbb{V}$. 
A configuration on $\mathbb{T}$ is an element of $\mathcal{C}^\mathbb{T}$, which is an assignment of a state in $\mathcal{C}$ to each vertex.  The state of the root $\rho$, denoted by $\sigma_\rho$, is
chosen according to an initial distribution $\pi$ on $\mathcal{C}$.  
This symbol is then propagated on the tree according to a probability transition matrix
$\mathbf{M}=(M_{ij})_{i, j \in \mathcal{C}}$, which functions as a noisy communication channel on each
edge. In other words2, for each
vertex $v$ having $u$ as its parent, the spin at $v$ is defined
according to the probabilities
$$
\mathbf{P}(\sigma_v=j\mid\sigma_u=i)=M_{i j}, \quad i, j \in \mathcal{C}.
$$

The problem of reconstruction is to analyze whether there exists a non-vanishing information on the letter transmitted by the
root, given all the symbols received at the vertices of the $n$th generation, as $n$ goes to infinity. Denote $\sigma(n)$ as the spins at distance
$n$ from the root and $\sigma^i(n)$ as $\sigma(n)$
conditioned on $\sigma_\rho = i$. In this paper, we use the following definition to mathematical formulate reconstructibility and we remark that more equivalent formulations can be seen in \cite{mossel2001reconstruction} and \cite{mossel2004survey}.
\begin{definition}
	We say that a model is \textup{\textbf{reconstructible}} on an infinite tree $\mathbb{T}$, if for some $i, j\in \mathcal{C}$
	$$
	\limsup_{n\to \infty}d_{TV}(\sigma^i(n), \sigma^j(n))>0,
	$$
	where $d_{TV}$ is the total variation distance. When the $\limsup$
	is $0$, we say that the model is \textup{\textbf{non-reconstructible}} on
	$\mathbb{T}$.
\end{definition}

\subsection{Existing results}
The reconstruction problem arises naturally in statistical physics, where the reconstruction threshold corresponds to the threshold for extremality of the infinite-volume Gibbs measure with
free boundary conditions (see \cite{georgii2011gibbs}).
The
reconstruction bound is known to have a crucial determination effect on the efficiency of the Glauber
dynamics on trees and random graphs (see \cite{berger2005glauber, martinelli2007fast, tetali2012phase}). It is specially worth mentioning that one of the classical techniques in tackling the reconstruction problem was initiated in \cite{chayes1986mean} in the context of spin-glasses.
The reconstructability is believed to play an
important role in a variety of other contexts including, but not limited
to, the following:  
phylogenetic reconstruction in evolutionary
biology (see \cite{mossel2004phase, daskalakis2006optimal, roch2006short}), communication theory in the study of noisy
computation (see \cite{evans2000broadcasting}), analogous investigations in the realm of network tomography (see \cite{bhamidi2010network}), reconstructability and distinguishability regarding the clustering problem in the stochastic block model (see \cite{mossel2013proof, mossel2014belief, neeman2014non}), analogous phase transition analysis using the cavity method in Bayesian inference (see \cite{ricci2018typology}), 
classification and labeling in a semi-supervised learning setting using deep algorithms on a closely related and well defined family of hierarchical generative models (see \cite{mossel2016deep}).
More applications can be seen in Section 1.4 of \cite{sly2009reconstruction} and Section 1.3 of \cite{liu2018tightness}, and the references therein.
 
Clearly, for any channel, the
reconstruction problem is closely related to $\lambda$, the second
largest eigenvalue in absolute value of the transition probability matrix $\mathbf{M}$. \cite{kesten1966additional,kesten1967limit} showed that
the reconstruction problem is solvable if $d\lambda^2>1$, which is known as the
Kesten-Stigum bound. However, when it comes to the case of larger noise, that is
$d\lambda^2 < 1$, retrieving root information from the symbols received at the $n$th generation, is a challenging problem whose solvability highly depends on the channel.
One important case is the binary channel, that is, the Ising model in statistical mechanics terminology, with the transition matrix
$$
\mathbf{M}= \frac{1}{2} \left(
\begin{array}{cc}
1+\theta & 1-\theta \\
1-\theta & 1+\theta \\
\end{array}
\right) + \frac{\Delta}{2}\left(
\begin{array}{cc}
-1 & 1 \\
-1 & 1 \\
\end{array}
\right),
$$
where $|\theta|+|\Delta|\leq 1$ and $\Delta$ is used to describe the deviation of $\mathbf{M}$
from the symmetric channel. 
For the binary symmetric
model, i.e. $\Delta=0$, it was shown in \cite{bleher1995purity} that the
reconstruction problem is solvable if and only if $d\lambda^2>1$. For the
binary asymmetric model,  i.e. $\Delta \neq 0$, \cite{mossel2001reconstruction, mossel2004survey}  showed that
the Kesten-Stigum bound is not the bound for reconstruction with sufficiently large asymmetry. The breakthrough result in \cite{borgs2006kesten} established the first tightness result of Keston-Stigum reconstruction bound in roughly a decade for a binary asymmetric channel on the
$d$-ary tree, provided that the asymmetry is sufficiently
small. \cite{liu2018large} gave a complete answer to the question of how small it needs to be to establish the tightness of the reconstruction threshold.

Another important case is the $q$-state symmetric channel, that is, the Potts model in statistical mechanics terminology, with the transition matrix
$$
\mathbf{M}=\left(
\begin{array}{cccccccc}
p_0 & p_1 & \cdots & p_1  \\
p_1 & p_0 & \cdots & p_1  \\
\vdots & \vdots & \ddots & \vdots \\
p_1 & p_1 & \cdots & p_0
\end{array}
\right)_{q\times q}.
$$
\cite{mossel2001reconstruction, mossel2004survey}  showed that
the Kesten-Stigum bound is not sharp in the
$q$-state symmetric
channel with sufficiently many characters.
\cite{sly2009reconstruction} proved the first exact reconstruction threshold in a nonbinary model, established the Kesten-Stigum bound for the $3$-state Potts model on regular trees of large degree,  and further showed that the Kesten-Stigum bound is not tight when $q \geq 5$, which confirms much of the picture conjectured earlier by \cite{mezard2006reconstruction}. Inspired by a popular Markov model of DNA sequence evolution, the K$80$ model (\cite{kimura1980simple}), to distinguish between transitions and transversions, \cite{liu2018tightness} analyzed the case that the transition matrix has two mutation classes and $q$ states in each class, with the transition matrix
\[\mathbf{M}=
\left( \begin{array}{@{}cccc|cccc@{}}
p_0 & p_1 & \cdots & p_1 & p_2 & \cdots & \cdots & p_2 \\
p_1 & p_0 & \cdots & p_1  & \vdots & \ddots &  & \vdots \\
\vdots & \vdots & \ddots & \vdots & \vdots &  & \ddots & \vdots \\
p_1 & p_1 & \cdots & p_0 & p_2 & \cdots & \cdots & p_2\\\hline
p_2 & \cdots & \cdots & p_2 & p_0 & p_1 & \cdots & p_1\\
\vdots & \ddots &  & \vdots & p_1 & p_0 & \cdots & p_1\\
\vdots &  & \ddots & \vdots & \vdots & \vdots & \ddots & \vdots\\
p_2 & \cdots & \cdots & p_2 & p_1 & p_1 & \cdots & p_0\\
\end{array} \right)_{2q\times 2q}.
\]
\cite{liu2018tightness} showed that when $q\geq
4$, for every $d$ the Kesten-Stigum bound is not tight, i.e. the
reconstruction is solvable for some $\lambda$ even if
$d\lambda^2<1$. 

\subsection{Motivation and the Main Theorem}
Deoxyribonucleic acid (DNA) is a molecule composed of two chains coiling around each other to form a double helix, which carries the genetic instructions used in the growth, development, functioning and reproduction of all known living organisms and many viruses. These two DNA strands are composed of simpler monomeric units called nucleotides, each of which is further composed of one of four nitrogen-containing nucleobases (guanine [G], cytosine [C], adenine [A] or thymine [T]).  A number of different Markov models of DNA sequence evolution were proposed and
 frequently used in molecular phylogenetics, a branch of phylogeny to analyze the genetic or hereditary molecular differences in order to gain information on an organism's evolutionary relationships.  Specifically, they are used during the calculation of likelihood of a tree
and used to estimate the evolutionary distance from the observed differences between sequences. For detailed descriptions of phylogenetic reconstruction and Markov models of DNA evolution, we refer to \cite{felsenstein2004inferring}.

In this paper, the objective model taken into
account is based on the F$81$ model (\cite{felsenstein1981evolutionary}), a classical DNA evolution model, whose base frequencies are allowed to vary from $0.25$ that is enforced in several other classical DNA evolution models. We allow a guanine-cytosine (G-C) content bias, where the G-C content refers to the percentage of nitrogenous bases on a DNA molecule that are either guanine or cytosine. We further follow the Chargaff's parity rule, that is, pairing nucleotides do have the same frequency on a single DNA strand, with G and C on the one hand while A and T on the other hand. Therefore, the four base frequencies can be expressed as a function of $\theta\in (0,1)$: 
$\pi_A= \pi_T  = \frac{\theta}2$ and $\pi_G= \pi_C  = \frac{1-\theta}2.$
In other word, we consider the rate matrix of $\{A, T, G, C\}$, or the configuration set $\{1, 2, 3, 4\}$, of the form
$$
\mathbf{Q}=\frac{1}2\left(
\begin{array}{cccc}
-2+\theta &\quad \theta &\quad 1-\theta &\quad 1-\theta
\\ 
\theta &\quad -2+\theta &\quad 1-\theta &\quad 1-\theta
\\
\theta &\quad \theta &\quad -1-\theta  &\quad 1-\theta
\\
\theta &\quad \theta &\quad 1-\theta
&\quad -1-\theta
\end{array}
\right).
$$
Set $\lambda=\exp\left({-\frac{v}{\frac12+\theta(1-\theta)}}\right)$, where $v$ is defined as the branch length measured in the expected number of substitutions per site, and then the corresponding probability transition matrix can be written as
$$
\mathbf{P}=\frac{1}2\left(
\begin{array}{cccc}
2\lambda+\theta(1-\lambda)&\quad \theta(1-\lambda)&\quad (1-\theta)(1-\lambda) &\quad (1-\theta)(1-\lambda) 
\\ 
\theta(1-\lambda)&\quad  2\lambda+\theta(1-\lambda)&\quad (1-\theta)(1-\lambda) &\quad (1-\theta)(1-\lambda) 
\\
\theta(1-\lambda)&\quad  \theta(1-\lambda)&\quad 2\lambda+(1-\theta)(1-\lambda) &\quad (1-\theta)(1-\lambda) 
\\
\theta(1-\lambda)&\quad  \theta(1-\lambda)&\quad (1-\theta)(1-\lambda)
&\quad 2\lambda+(1-\theta)(1-\lambda) 
\end{array}
\right).
$$
It is clear that the second eigenvalue of the channel
$\mathbf{P}$ is $\lambda$. Because non-reconstruction always holds for $d|\lambda|^2>1$, without
loss of generality, it would be convenient to presume $ d|\lambda|^2\leq 1$ in the following context.

We focus on the $4\times 4$-state probability transition matrix allowing $\theta \neq 1/2$. Note that, when $\theta=1/2$, the model degenerates to the symmetric $q$-state Potts model with $q=4$, whose reconstructibility is an open problem, while \cite{sly2009reconstruction} gave perfect answers on the reconstructibility with $q=3$ and $q\geq 5$. 
The asymmetry and community effects can be seen in the following: (1) One state stays unchanged with probability $\lambda+\frac{\theta(1-\lambda)}2$ if that state is in $\{1,2\}$, but with probability $\lambda+\frac{(1-\theta)(1-\lambda)}2 $ if that state is in $\{3,4\}$; (2) For the community $\{1,2\}$ the probability to transfer from one state to the other is $\frac{\theta(1-\lambda)}2$, while for the community $\{3,4\}$ the probability to transfer from one state to the other is $\frac{(1-\theta)(1-\lambda)}2$ for the case in $\{3,4\}$; (3) The transition probabilities from $\{1,2\}$ to $\{3,4\}$ are given by $\frac{(1-\theta)(1-\lambda)}2$, while the transition probabilities from $\{3,4\}$ to $\{1,2\}$ are given by $\frac{\theta(1-\lambda)}2$.

The main result is given below and its rigorous proof is provided in Section \ref{Sec:Proof_of_Main_Theorem}.
\begin{Main Theorem}
	\label{reconstruction} 
	When $\theta\in \left(0,\frac{3-\sqrt{3}}6\right)\bigcup \left(\frac{3+\sqrt{3}}{6},1\right)$, for
	every $d$ the Kesten-Stigum bound is not sharp. 
\end{Main Theorem}


\subsection{Proof sketch}
The ideas and techniques used to prove the Main Theorem can be seen in the following. One standard to classify reconstruction and non-reconstruction is to analyze the difference, between the probability of giving a correct guess of the root given the spins $\sigma(n)$ at distance $n$ from the root, and the probability of guessing the root according to the stationary initial distribution. Unlike the symmetric models, the model under investigation has two distinct base frequencies $\theta / 2$ and $(1-\theta) / 2$, therefore we need to analyze two different quantities $x_{n;\theta}$ (probability of giving a correct guess of the root in $\{1, 2\}$ minus $\theta / 2$) and $x_{n;1-\theta}$ (probability of giving a correct guess of the root in $\{3, 4\}$ minus $(1-\theta) / 2$). Since non-reconstruction means that the mutual information between the root and the spins at distance $n$ goes to $0$ as $n$ tends to infinity, here it can be established that non-reconstruction is equivalent to
$$
\lim_{n\to \infty}x_{n;\theta}=\lim_{n\to\infty}x_{n;1-\theta}=0.
$$
Furthermore, it is necessary to consider the quantity $y_{n;\theta}$  (resp. $y_{n;1-\theta}$) which corresponds to the probability of giving a wrong guess but right group $\{1, 2\}$ (resp. $\{3, 4\}$), and the quantity $z_{n;1-\theta}$ (resp. $z_{n;\theta}$) which corresponds to the probability of giving a wrong guess and even wrong group $\{3, 4\}$ (resp. $\{1, 2\}$). Together with their second moment forms, we have to take care of much more objective quantities than~\cite{sly2009reconstruction},~\cite{liu2018large} and~\cite{liu2018tightness}.

Through carefully analyzing the relation between the $n$th and the $(n+1)$th level, and fully taking advantage of the Markov random field property and the symmetries incorporated in the model and the tree structure, we establish the distributional recursion and moment recursion. We show that the interactions between spins become very weak if they are sufficiently far away from each other, and prove that $x_{n;\theta}$ is small and the decrease from $x_{n;\theta}$ to $x_{n+1;\theta}$ is never too large. Consequently, we are able to derive the concentration estimates and achieve the following asymptotic $4$-dimensional second order nonlinear dynamical system: 
\begin{equation*}
\begin{split}
\label{eq:XZ_dynamics}
\left\{\begin{array}{llll}
x_{n+1;\theta}&\approx &d\lambda^2x_{n;\theta}+\frac{d(d-1)}{2}\lambda^4\vast[\left(-6+\frac{2(1-\theta)}{\theta}
+2\theta\right)x_{n;\theta}^2+\left(-\frac{4\theta}{1-\theta}-16\right)(-z_{n;1-\theta})^2-(4\theta-16)x_{n;\theta}(-z_{n;1-\theta})\\
&&\quad\quad\quad\quad\quad\quad\quad\quad\quad\;\;+\theta\left(x_{n;1-\theta}^2+y_{n;1-\theta}^2\right)\vast],
\\
\
\\
-z_{n+1;1-\theta}&\approx &d\lambda^2(-z_{n;1-\theta})+\frac{d(d-1)}{2}\lambda^4\vast[\frac{2(1-\theta)^2}{\theta} x_{n;\theta}^2
+\left(-8-\frac{4\theta}{1-\theta}+\frac{4(1-\theta)}{\theta}\right)(-z_{n;1-\theta})^2-4\frac{(1-\theta)^2}{\theta}x_{n;\theta}(-z_{n;1-\theta})\\
&&\quad\quad\quad\quad\quad\quad\quad\quad\quad\quad\quad+\theta\left(x_{n;1-\theta}^2+y_{n;1-\theta}^2\right)\vast].
\end{array}
\right.
\end{split}
\end{equation*}
Methods in the previous works are to analyze the stability of fixed points of the $x$-dynamics alone, by showing that if the quadratic terms are strictly positive, given the fact that $x_{n;\theta}$ is nonnegative, it will not goes to zero as $n$ goes to infinity. However, a closer look at the above quadratic terms in the $x$-dynamics reveals that it is hard to obtain the desired results. Therefore, we turn to focus on analyzing the $z$-dynamics while fully take into consideration of the coupled relationship between $x_{n;\theta}$ and $z_{n;1-\theta}$, and conclude that when $\theta\in \left(0,\frac{3-\sqrt{3}}6\right)\bigcup \left(\frac{3+\sqrt{3}}{6},1\right)$, even if $d\lambda^2<1$ for some $\lambda$,  $z_{n;1-\theta}$ does not converge to $0$ and then $x_{n;\theta} (\geq-z_{n;1-\theta}\geq 0)$ does not converge to $0$. That is, there is reconstruction beyond the Kesten-Stigum bound. 

\subsection{Organization of the paper}
The rest of the paper is organized as follows. In Section \ref{Sec:Preliminary_Results}, we give detailed definitions and interpretations, conduct preliminary analyses, and then provide an equivalent condition for non-reconstruction.  In Section \ref{Sec:Distributional_Recursion}, we develop the second order
recursive relations associated with $x_{n+1;\theta}$ and $z_{n+1;1-\theta}$. In-depth concentration results are established in Section \ref{Sec:Concentration_Analysis}.
A complete proof of the Main Theorem is given in Section \ref{Sec:Proof_of_Main_Theorem}.

\section{Preparation}
\label{Sec:Preliminary_Results}
\subsection{Notations}
\label{Sec:Notations}
Recall that the stationary distribution $\pi=(\pi_1, \pi_2, \pi_3, \pi_4)$ of
$\mathbf{P}$ is given by
$$
\pi_1=\pi_2=\frac{\theta}{2}
\quad\textup{and}\quad
\pi_3=\pi_4=\frac{1-\theta}{2}.
$$
Let $u_1,\ldots,u_d$
be the children of root $\rho$ and $\mathbb{T}_v$ be the subtree of
descendants of $v\in \mathbb{T}$. Denote the
$n$th level of the tree by $L_n=\{v\in\mathbb{V}: d(\rho, v)=n\}$ with $d(\cdot,
\cdot)$ being the graph distance on $\mathbb{T}$.
Let $\sigma(n)$ and $\sigma_j(n)$ denote
the spins on $L_n$ and $L_n\cap \mathbb{T}_{u_j}$ respectively. For
a configuration $A$ on the spins of $L_n$, define the posterior function by
$$
f_n(i, A)=\mathbf{P}(\sigma_\rho=i\mid\sigma(n)=A), \quad i=1, 2, 3, 4.
$$
By the recursive nature of the tree, for a configuration $A$ on spins in 
$L(n+1) \cap \mathbb{T}_{u_j}$, we can give an equivalent form of
the previous one
$$
f_n(i, A)=\mathbf{P}(\sigma_{u_j}=i\mid\sigma_j(n+1)=A), \quad i=1, 2, 3, 4, \quad j=1,\cdots,d.
$$
Define $X_i$ as the posterior probability that $\sigma_{\rho}=i$ given the random configuration $\sigma(n)$ on the spins in $L_n$, that is
$$
X_i=X_i(n)=f_n(i, \sigma(n)), \quad i=1, 2, 3, 4.
$$
Then one has
$$
X_1(n)+X_2(n)+X_3(n)+X_4(n)=1
$$
and
$$
\mathbf{E}(X_1)=\mathbf{E}(X_2)=\frac\theta2,\quad \mathbf{E}(X_3)=\mathbf{E}(X_4)=\frac{1-\theta}2.
$$
Define $Y_{ij}$ as the posterior probability that $\sigma_{u_j}=i$ given the random configuration $\sigma^1_j (n + 1)$ on spins in 
$L(n+1) \cap \mathbb{T}_{u_j}$, that is
$$
Y_{ij}=Y_{ij}(n)=f_n(i, \sigma_j^1(n+1)), \quad i=1, 2, 3, 4, \quad j=1,\cdots,d,
$$
where the random variables $\{Y_{ij}\}$ are independent and identically distributed apparently. 
At last, we define the following moment variables, which will serve as the main quantities under investigation, 
$$
x_{n;\theta}=\mathbf{E}\left(f_n(1, \sigma^1(n))-\frac{\theta}{2}\right), \quad
y_{n;\theta}=\mathbf{E}\left(f_n(2, \sigma^1(n))-\frac{\theta}{2}\right), \quad
z_{n;\theta}=\mathbf{E}\left(f_n(1, \sigma^3(n))-\frac{\theta}{2}\right),
$$
$$
u_{n;\theta}=\mathbf{E}\left(f_n(1, \sigma^1(n))-\frac{\theta}{2}\right)^2,\quad
v_{n;\theta}=\mathbf{E}\left(f_n(2, \sigma^1(n))-\frac{\theta}{2}\right)^2,\quad
w_{n;\theta}=\mathbf{E}\left(f_n(1, \sigma^3(n))-\frac{\theta}{2}\right)^2,
$$
and 
$$
x_{n;1-\theta}=\mathbf{E}\left(f_n(3, \sigma^3(n))-\frac{1-\theta}{2}\right), \quad
y_{n;1-\theta}=\mathbf{E}\left(f_n(4, \sigma^3(n))-\frac{1-\theta}{2}\right), \quad
z_{n;1-\theta}=\mathbf{E}\left(f_n(3, \sigma^1(n))-\frac{1-\theta}{2}\right),
$$
$$
u_{n;1-\theta}=\mathbf{E}\left(f_n(3, \sigma^3(n))-\frac{1-\theta}{2}\right)^2,\quad
v_{n;1-\theta}=\mathbf{E}\left(f_n(4, \sigma^3(n))-\frac{1-\theta}{2}\right)^2,\quad
w_{n;1-\theta}=\mathbf{E}\left(f_n(3, \sigma^1(n))-\frac{1-\theta}{2}\right)^2.
$$

\subsection{Preliminary analyses}
Let us firstly establish some important lemmas which will be used frequently in the sequel.
\begin{lemma}
	\label{lemma1}
	For any $n\in \mathbb{N}\bigcup\{0\}$, we have
	\begin{enumerate}[(a)]
		\item $x_{n;\theta}=\frac{2}{\theta}\mathbf{E}\left(X_1(n)-\frac\theta2\right)^2=u_{n;\theta}+v_{n;\theta}+2\frac{1-\theta}{\theta}w_{n;\theta}\geq 0$.
		\item $z_{n;1-\theta}=-\frac{x_{n;\theta}+y_{n;\theta}}2\leq 0, \quad x_{n;\theta}+z_{n;1-\theta}\geq 0$;\\ $z_{n;\theta}=-\frac{x_{n;1-\theta}+y_{n;1-\theta}}2\leq 0, \quad x_{n;1-\theta}+z_{n;\theta}\geq 0$.
		\item $\theta z_{n;1-\theta}=(1-\theta)z_{n;\theta}.$
	\end{enumerate}
\end{lemma}
\begin{proof}
	\begin{enumerate}[(a)]
		\item By the law of total probability and Bayes' theorem, we have
		\begin{eqnarray*}
			\mathbf{E}f_n(1, \sigma^1(n))
			&=&\sum_A f_n(1,A)\mathbf{P}(\sigma(n)=A\mid\sigma_\rho=1)
			\\
			&=&\frac2{\theta}\sum_A\mathbf{P}(\sigma_\rho=1\mid\sigma(n)=A)\mathbf{P}(\sigma(n)=A)f_n(1,A)
			\\
			&=&\frac2{\theta}\sum_Af_n^2(1,A)\mathbf{P}(\sigma(n)=A)
			\\
			&=&\frac2{\theta}\mathbf{E}(X_1^2).
		\end{eqnarray*}
		Then it follows from the fact $\mathbf{E}(X_1)=\frac\theta2$ that
		$$
		x_{n;\theta}=\frac{2}{\theta}\left(\mathbf{E}(X_1^2)-\left(\frac{\theta}{2}\right)^2\right)=\frac{2}{\theta}\mathbf{E}\left(X_1-\frac\theta2\right)^2.
		$$
		Furthermore, by the law of total expectation, we have
		\begin{eqnarray*}
			x_{n;\theta}&=&\frac{2}{\theta}\mathbf{E}\left(X_1-\frac{\theta}{2}\right)^2
			\\
			&=&\frac{2}{\theta}\sum_{i=1}^4\mathbf{E}\left(\left(X_1-\frac\theta2\right)^2\mid\sigma_\rho=i\right)\mathbf{P}(\sigma_\rho=i)
			\\
			&=&\frac{2}{\theta}\left[\mathbf{P}(\sigma_\rho=1)\mathbf{E}\left(f_n(1, \sigma^1(n))-\frac\theta2\right)^2+\mathbf{P}(\sigma_\rho=2)\mathbf{E}\left(f_n(1, \sigma^2(n))-\frac\theta2\right)^2\right.
			\\
			&&\quad+\left.\mathbf{P}(\sigma_\rho=3)\mathbf{E}\left(f_n(1, \sigma^3(n))-\frac\theta2\right)^2
			+\mathbf{P}(\sigma_\rho=4)\mathbf{E}\left(f_n(1, \sigma^4(n))-\frac\theta2\right)^2\right]
			\\
			&=&u_{n;\theta}+v_{n;\theta}+2\frac{1-\theta}{\theta}w_{n;\theta}.
		\end{eqnarray*}

		\item Similarly, we have 
		\begin{eqnarray*}
			y_{n;\theta}+\frac{\theta}{2}
			&=&\sum_A f_n(2,A)\mathbf{P}(\sigma(n)=A\mid\sigma_\rho=1)\\
			&=&\frac2{\theta}\sum_Af_n(1,A)f_n(2,A)\mathbf{P}(\sigma(n)=A)\\
			&=&\frac2\theta\mathbf{E}\left(X_1X_2\right),
		\end{eqnarray*}
		and then
		\begin{equation}
		\label{y}
		y_{n;\theta}=\frac2\theta\mathbf{E}\left(X_1-\frac\theta2\right)\left(X_2-\frac\theta2\right).
		\end{equation}
		It follows from the Cauchy-Schwarz inequality that
		$$
		\left[\mathbf{E}\left(X_1-\frac\theta2\right)\left(X_2-\frac\theta2\right)\right]^2\leq\mathbf{E}\left(X_1-\frac\theta2\right)^2\mathbf{E}\left(X_2-\frac\theta2\right)^2,
		$$
		which implies 
		\begin{equation}
		\label{cauchy}
		\left(\frac\theta2y_{n;\theta}\right)^2\leq\left(\frac\theta2x_{n;\theta}\right)^2, \quad\textup{i.e. } -x_{n;\theta}\leq y_{n;\theta}\leq x_{n;\theta}.
		\end{equation}
		By the definitions of $x_{n;\theta}$, $y_{n;\theta}$ and $z_{n;1-\theta}$, 
		we know that
		$z_{n;1-\theta}=-\frac{x_{n;\theta}+y_{n;\theta}}2$, and thus \eqref{cauchy} implies $$x_{n;\theta}+z_{n;1-\theta}=x_{n;\theta}-\frac{x_{n;\theta}+y_{n;\theta}}{2}=\frac{x_{n;\theta}-y_{n;\theta}}{2}\geq 0 \quad \text{and} \quad z_{n;1-\theta}\leq 0.$$
		An analogous proof of $$z_{n;\theta}=-\frac{x_{n;1-\theta}+y_{n;1-\theta}}2\leq 0 \quad \text{and} \quad x_{n;1-\theta}+z_{n;\theta}\geq 0$$
		can be easily carried out. 
		\item Similarly, we have 	
		\begin{equation}
		\label{zn}
		z_{n;1-\theta}=\mathbf{E}\left(f_n(3, \sigma^1(n))-\frac{1-\theta}2\right)=\frac2{\theta}\mathbf{E}\left(X_1X_3\right)-\frac{1-\theta}2
		=\frac2{\theta}\mathbf{E}\left(X_1-\frac{\theta}{2}\right)\left(X_3-\frac{1-\theta}{2}\right)
		\end{equation}
		and then
		$$
		\theta z_{n;1-\theta}=2\mathbf{E}\left(X_1-\frac{\theta}{2}\right)\left(X_3-\frac{1-\theta}{2}\right)=(1-\theta)z_{n;\theta}.
		$$ 
	\end{enumerate}
\end{proof}

\begin{lemma}
	\label{xnun} For any $n\in \mathbb{N}\cup\{0\}$, we have
	\begin{enumerate}[(a)]
		
		\item $\mathbf{E}\left(f_n(1, \sigma^1(n))-\frac{\theta}{2}\right)\left(f_n(2, \sigma^1(n))-\frac{\theta}{2}\right)=\frac{\theta}{2}y_{n;\theta}+\left(v_{n;\theta}-\frac{\theta}{2}x_{n;\theta}\right)$.
		
		\item $\mathbf{E}\left(f_n(1, \sigma^1(n))-\frac{\theta}{2}\right)\left(f_n(3, \sigma^1(n))-\frac{1-\theta}{2}\right)=\frac{\theta}{2}z_{n;1-\theta}-\frac12\left(u_{n;\theta}-\frac{\theta}{2}x_{n;\theta}\right)-\frac12\left(v_{n;\theta}-\frac{\theta}{2}x_{n;\theta}\right)$.
		
		\item $\mathbf{E}\left(f_n(2, \sigma^1(n))-\frac{\theta}{2}\right)\left(f_n(3, \sigma^1(n))-\frac{1-\theta}{2}\right)=\frac{\theta}{2}z_{n;1-\theta}-\left(v_{n;\theta}-\frac{\theta}{2}x_{n;\theta}\right)$.

		\item $\mathbf{E}\left(f_n(3, \sigma^1(n))-\frac{1-\theta}{2}\right)\left(f_n(4, \sigma^1(n))-\frac{1-\theta}{2}\right)
		=\frac{(1-\theta) }{2}y_{n;1-\theta}+\frac12\left(u_{n;\theta}-\frac{\theta }{2}x_{n;\theta}\right)+\frac32\left(v_{n;\theta}-\frac{\theta }{2}x_{n;\theta}\right)-\left(w_{n;1-\theta}-\frac{1-\theta}{2}x_{n;1-\theta}\right)$.
		
		\item $\mathbf{E}\left(f_n(1, \sigma^3(n))-\frac{\theta}{2}\right)\left(f_n(2, \sigma^3(n))-\frac{\theta}{2}\right)=\frac{\theta}{2}y_{n;\theta}-\frac{\theta}{1-\theta}\left(v_{n;\theta}-\frac{\theta }{2}x_{n;\theta}\right)$.
	\end{enumerate}
\end{lemma}
\begin{proof}
	\begin{enumerate}[(a)]
		\item By the law of total probability, one has
		\begin{equation*}
		\begin{aligned}
		\mathbf{E}f_n(1,\sigma^1(n))f_n(2,\sigma^1(n))&=\sum_{A}\mathbf{P}(\sigma_\rho=1\mid\sigma(n)=A)\mathbf{P}(\sigma_\rho=2\mid\sigma(n)=A)\mathbf{P}(\sigma(n)=A\mid\sigma_\rho=1)
		\\
		&=\sum_{A}\left[\mathbf{P}(\sigma_\rho=2\mid\sigma(n)=A)\right]^2\mathbf{P}(\sigma(n)=A\mid\sigma_\rho=1)
		\\
		&=\mathbf{E}\left(f_n(2,\sigma^1(n))\right)^2,
		\end{aligned}
		\end{equation*}
		and therefore,
		\begin{equation*}
		\begin{aligned}
		\mathbf{E}\left(f_n(1,\sigma^1(n))-\frac{\theta}{2}\right)\left(f_n(2,\sigma^1(n))-\frac{\theta}{2}\right)&=\mathbf{E}\left(f_n(2,\sigma^1(n))-\frac{\theta}{2}\right)^2+\frac{\theta(y_{n;\theta}-x_{n;\theta})}{2}\\
		&=v_{n;\theta}+\frac{\theta}{2}\left(y_{n;\theta}-x_{n;\theta}\right)
		\\
		&=\frac{\theta}{2}y_{n;\theta}+\left(v_{n;\theta}-\frac{\theta}{2}x_{n;\theta}\right).
		\end{aligned}
		\end{equation*}
		\item Similarly, we can achieve that
		\begin{equation*}
		\begin{aligned}
		\mathbf{E}\left(f_n(1, \sigma^1(n))-\frac{\theta}{2}\right)\left(f_n(3, \sigma^1(n))-\frac{1-\theta}{2}\right)&=\frac{\theta z_{n;1-\theta}}{2}-\frac12\left(u_{n;\theta}-\frac{\theta x_{n;\theta}}{2}\right)-\frac12\left(v_{n;\theta}-\frac{\theta x_{n;\theta}}{2}\right)\\
		&=\frac{\theta}{2}(x_{n;\theta}+z_{n;1-\theta})-\frac{u_{n;\theta}+v_{n;\theta}}{2}
		\\
		&=\frac{\theta}{2}z_{n;1-\theta}-\frac12\left(u_{n;\theta}-\frac{\theta}{2}x_{n;\theta}\right)-\frac12\left(v_{n;\theta}-\frac{\theta}{2}x_{n;\theta}\right).
		\end{aligned}
		\end{equation*}
		
		\item Note that
		$$
		\left(f_n(2,\sigma^1(n))-\frac{\theta}{2}\right)+\left(f_n(3,\sigma^1(n))-\frac{1-\theta}{2}\right)=-\left(f_n(1,\sigma^1(n))-\frac{\theta}{2}\right)-\left(f_n(4,\sigma^1(n))-\frac{1-\theta}{2}\right),
		$$
		and by taking square of both sides and then expectation, one has
		\begin{equation*}
		\begin{aligned}
		\mathbf{E}\left(f_n(2, \sigma^1(n))-\frac{\theta}{2}\right)\left(f_n(3, \sigma^1(n))-\frac{1-\theta}{2}\right)
		=&\frac{u_{n;\theta}-v_{n;\theta}}2+\mathbf{E}\left(f_n(1, \sigma^1(n))-\frac{\theta}{2}\right)\left(f_n(4, \sigma^1(n))-\frac{1-\theta}{2}\right)
		\\
		=&\frac{u_{n;\theta}-v_{n;\theta}}2+\frac{\theta z_{n;1-\theta}}{2}-\frac12\left(u_{n;\theta}-\frac{\theta x_{n;\theta}}{2}\right)-\frac12\left(v_{n;\theta}-\frac{\theta x_{n;\theta}}{2}\right)
		\\
		=&\frac{\theta }{2}\left( x_{n;\theta}+z_{n;1-\theta}\right)-v_{n;\theta}
		\\
		=&\frac{\theta}{2}z_{n;1-\theta}-\left(v_{n;\theta}-\frac{\theta}{2}x_{n;\theta}\right).
		\end{aligned}
		\end{equation*}
		\item	Similarly, we have
		\begin{eqnarray*}
			&&\mathbf{E}\left(f_n(3, \sigma^1(n))-\frac{1-\theta}{2}\right)\left(f_n(4, \sigma^1(n))-\frac{1-\theta}{2}\right)\\&=&-\mathbf{E}\left(f_n(3, \sigma^1(n))-\frac{1-\theta}{2}\right)\left(f_n(1, \sigma^1(n))-\frac{\theta}{2}\right)-\mathbf{E}\left(f_n(3, \sigma^1(n))-\frac{1-\theta}{2}\right)\left(f_n(2, \sigma^1(n))-\frac{\theta}{2}\right)-\mathbf{E}\left(f_n(3, \sigma^1(n))-\frac{1-\theta}{2}\right)^2
			\\
			&=&-\frac{\theta z_{n;1-\theta}}{2}+\frac12\left(u_{n;\theta}-\frac{\theta x_{n;\theta}}{2}\right)+\frac12\left(v_{n;\theta}-\frac{\theta x_{n;\theta}}{2}\right)-\frac{\theta z_{n;1-\theta}}{2}+\left(v_{n;\theta}-\frac{\theta x_{n;\theta}}{2}\right)-w_{n;1-\theta}
			\\
			&=&\frac{(1-\theta) y_{n;1-\theta}}{2}+\frac12\left(u_{n;\theta}-\frac{\theta x_{n;\theta}}{2}\right)+\frac32\left(v_{n;\theta}-\frac{\theta x_{n;\theta}}{2}\right)-\left(w_{n;1-\theta}-\frac{(1-\theta)x_{n;1-\theta}}{2}\right).
		\end{eqnarray*}
		
		\item	By equation \eqref{y}, we have
		\begin{eqnarray*}
			\frac{\theta}{2}y_{n;\theta}&=&\mathbf{E}\left(X_1-\frac\theta2\right)\left(X_2-\frac\theta2\right)
			\\
			&=&\theta\mathbf{E}\left(f_n(1, \sigma^1(n))-\frac{\theta}{2}\right)\left(f_n(2, \sigma^1(n))-\frac{\theta}{2}\right)+(1-\theta)\mathbf{E}\left(f_n(1, \sigma^3(n))-\frac{\theta}{2}\right)\left(f_n(2, \sigma^3(n))-\frac{\theta}{2}\right)
			\\
			&=&\frac{\theta^2}{2}y_{n;\theta}+\theta\left(v_{n;\theta}-\frac{\theta x_{n;\theta}}{2}\right)+(1-\theta)\mathbf{E}\left(f_n(1, \sigma^3(n))-\frac{\theta}{2}\right)\left(f_n(2, \sigma^3(n))-\frac{\theta}{2}\right),
		\end{eqnarray*}
		which implies
		$$
		\mathbf{E}\left(f_n(1, \sigma^3(n))-\frac{\theta}{2}\right)\left(f_n(2, \sigma^3(n))-\frac{\theta}{2}\right)=\frac{\theta}{2}y_{n;\theta}-\frac{\theta}{1-\theta}\left(v_{n;\theta}-\frac{\theta x_{n;\theta}}{2}\right).
		$$
	\end{enumerate}
\end{proof}


Recall that $Y_{ij}(n) = f_n\left(i, \sigma^1_j (n + 1)\right)$ is simply the posterior probability that $\sigma_{u_j}=i$ given the random configuration $\sigma^1_j (n + 1)$ on spins in 
$L(n+1) \cap \mathbb{T}_{u_j}$.
By the symmetry of the model, the random vectors $(Y_{ij})_{i=1}^{4}$
are independent. The central moments of $Y_{ij}$ would play a key role in further analysis, and therefore it is necessary to figure them out in the first place.
 
\begin{lemma}
	\label{lemma:Yproperties} For each $1\leq j\leq d$, we have
	\begin{enumerate}[(a)]
	\item $\mathbf{E}\left(Y_{1j}(n)-\frac{\theta}{2}\right)=\lambda x_{n;\theta}$.
	
	\item $\mathbf{E}\left(Y_{2j}(n)-\frac{\theta}{2}\right)=\lambda y_{n;\theta}$.
	
	\item $\mathbf{E}\left(Y_{ij}(n)-\frac{1-\theta}{2}\right)=\lambda z_{n;1-\theta}, \quad i=3,4.
	$
	
	\item $\!
	\begin{aligned}[t]
	\mathbf{E}\left(Y_{1j}(n)-\frac{\theta}{2}\right)^2
	=\frac{\theta}{2}x_{n;\theta}+\lambda\left(u_{n;\theta}-\frac{\theta}{2}x_{n;\theta}\right).
	\end{aligned}
	$

	\item $
	\!
	\begin{aligned}[t]
	\mathbf{E}\left(Y_{2j}(n)-\frac{\theta}{2}\right)^2
	=\frac{\theta}{2}x_{n;\theta}+\lambda\left(v_{n;\theta}-\frac{\theta}{2}x_{n;\theta}\right).
	\end{aligned}
	$
	
	\item 
	$\mathbf{E}\left(Y_{ij}(n)-\frac{1-\theta}{2}\right)^2
	=\frac{1-\theta}{2}x_{n;1-\theta}+\lambda\left(w_{n;1-\theta}-\frac{1-\theta}{2}x_{n;1-\theta}\right), \quad i=3,4.
	$
	
	\item $
	\!
	\begin{aligned}[t]
	\mathbf{E}\left(Y_{1j}(n)-\frac{\theta}{2}\right)\left(Y_{2j}(n)-\frac{\theta}{2}\right)
	=\frac{\theta}{2}y_{n;\theta}+\lambda\left(v_{n;\theta}-\frac{\theta}{2}x_{n;\theta}\right).
	\end{aligned}
	$
	
	\item 
$\mathbf{E}\left(Y_{1j}(n)-\frac{\theta}{2}\right)\left(Y_{ij}(n)-\frac{1-\theta}{2}\right)=\frac{\theta z_{n;1-\theta}}{2}-\frac\lambda2\left(u_{n;\theta}-\frac{\theta}{2}x_{n;\theta}\right)-\frac\lambda2\left(v_{n;\theta}-\frac{\theta}{2}x_{n;\theta}\right), \quad i=3,4$.
	
	\item 
	$\mathbf{E}\left(Y_{2j}(n)-\frac{\theta}{2}\right)\left(Y_{ij}(n)-\frac{1-\theta}{2}\right)
	=\frac{\theta z_{n;1-\theta}}{2}-\lambda\left(v_{n;\theta}-\frac{\theta}{2}x_{n;\theta}\right), \quad i=3,4.
	$
	
	\item $
	\!
	\begin{aligned}[t]
	\mathbf{E}\left(Y_{3j}(n)-\frac{1-\theta}{2}\right)\left(Y_{4j}(n)-\frac{1-\theta}{2}\right)
	=&\frac{1-\theta}{2}y_{n;1-\theta}+\frac\lambda2\left(u_{n;\theta}-\frac{\theta}{2}x_{n;\theta}\right)+\frac{3\lambda}2\left(v_{n;\theta}-\frac{\theta}{2}x_{n;\theta}\right)-\lambda\left(w_{n;1-\theta}-\frac{1-\theta}{2}x_{n;1-\theta}\right).
	\end{aligned}
	$
\end{enumerate}
\end{lemma}
	
\begin{proof} In the following, we only prove some results in Lemma \ref{lemma:Yproperties} and the rest can be shown analogously.
	\begin{enumerate}[(a)]
	\item $
	\!
	\begin{aligned}[t]
	\mathbf{E}\left(Y_{1j}(n)-\frac{\theta}{2}\right)
	&=p_{11}\mathbf{E}\left(f_n(1,\sigma^1(n))-\frac{\theta}{2}\right)+p_{12}\mathbf{E}\left(f_n(1, \sigma^2(n))-\frac{\theta}{2}\right)+p_{13}\mathbf{E}\left(f_n(1, \sigma^3(n))-\frac{\theta}{2}\right)
+p_{14}\mathbf{E}\left(f_n(1, \sigma^4(n))-\frac{\theta}{2}\right)
	\\
	&=\left(\lambda+\frac{\theta(1-\lambda)}{2}\right)x_{n;\theta}+\frac{\theta(1-\lambda)}{2}y_{n;\theta}+(1-\theta)(1-\lambda)z_{n;\theta}
	\\
	&=\lambda x_{n;\theta}.
	\end{aligned}
	$
	
	\item $
	\!
	\begin{aligned}[t]
	\mathbf{E}\left(Y_{2j}(n)-\frac{\theta}{2}\right)
	&=p_{11}\mathbf{E}\left(f_n(2,\sigma^1(n))-\frac{\theta}{2}\right)+p_{12}\mathbf{E}\left(f_n(2, \sigma^2(n))-\frac{\theta}{2}\right)+p_{13}\mathbf{E}\left(f_n(2, \sigma^3(n))-\frac{\theta}{2}\right)
+p_{14}\mathbf{E}\left(f_n(2, \sigma^4(n))-\frac{\theta}{2}\right)
	\\
	&=\left(\lambda+\frac{\theta(1-\lambda)}{2}\right)y_{n;\theta}+\frac{\theta(1-\lambda)}{2}x_{n;\theta}+(1-\theta)(1-\lambda)z_{n;\theta}
	\\
	&=\lambda y_{n;\theta}.
	\end{aligned}
	$
	
	\item
	It follows immediately from the identity $\sum_{i=1}^{4}Y_{ij}(n)=1$ that, for $i=3,4$
	$$
	\mathbf{E}\left(Y_{ij}(n)-\frac{1-\theta}{2}\right)=-\frac{1}{2}\sum_{i=1}^2\mathbf{E}\left(Y_{ij}(n)-\frac{\theta}{2}\right)=\lambda z_{n;1-\theta}.
	$$
	
    \setItemnumber{8}
	\item  
	By equation \eqref{zn}, we have
	\begin{eqnarray*}
		\frac{\theta}{2}z_{n;1-\theta}&=&\mathbf{E}\left(X_1-\frac\theta2\right)\left(X_3-\frac{1-\theta}2\right)
		\\
		&=&\frac\theta2\mathbf{E}\left(f_n(1, \sigma^1(n))-\frac{\theta}{2}\right)\left(f_n(3, \sigma^1(n))-\frac{1-\theta}{2}\right)+\frac\theta2\mathbf{E}\left(f_n(1, \sigma^2(n))-\frac{\theta}{2}\right)\left(f_n(3, \sigma^2(n))-\frac{1-\theta}{2}\right)
		\\
		&&+\frac{1-\theta}2\mathbf{E}\left(f_n(1, \sigma^3(n))-\frac{\theta}{2}\right)\left(f_n(3, \sigma^3(n))-\frac{1-\theta}{2}\right)
		+\frac{1-\theta}2\mathbf{E}\left(f_n(1, \sigma^4(n))-\frac{\theta}{2}\right)\left(f_n(3, \sigma^4(n))-\frac{1-\theta}{2}\right)
	\end{eqnarray*}
	and then
	\begin{eqnarray*}
		&&\mathbf{E}\left(f_n(1, \sigma^3(n))-\frac{\theta}{2}\right)\left(f_n(3, \sigma^3(n))-\frac{1-\theta}{2}\right)
		+\mathbf{E}\left(f_n(1, \sigma^4(n))-\frac{\theta}{2}\right)\left(f_n(3, \sigma^4(n))-\frac{1-\theta}{2}\right)
		\\
		&=&\theta z_{n;1-\theta}+\frac{\theta}{2(1-\theta)}\left[\left(u_{n;\theta}-\frac{\theta}{2}x_{n;\theta}\right)+3\left(v_{n;\theta}-\frac{\theta}{2}x_{n;\theta}\right)\right].
	\end{eqnarray*}
	Therefore, for $i=3,4$, Lemma \ref{xnun} implies 
	\begin{equation*}
	\begin{split}
	&\mathbf{E}\left(Y_{1j}(n)-\frac{\theta}{2}\right)\left(Y_{ij}(n)-\frac{1-\theta}{2}\right)\\
	=&p_{11}\mathbf{E}\left(f_n(1,\sigma^1(n))-\frac{\theta}{2}\right)\left(f_n(3, \sigma^1(n))-\frac{1-\theta}{2}\right)+p_{12}\mathbf{E}\left(f_n(1, \sigma^2(n))-\frac{\theta}{2}\right)\left(f_n(3, \sigma^2(n))-\frac{1-\theta}{2}\right)\\
	&+p_{13}\mathbf{E}\left(f_n(1, \sigma^3(n))-\frac{\theta}{2}\right)\left(f_n(3, \sigma^3(n))-\frac{1-\theta}{2}\right)+p_{14}\mathbf{E}\left(f_n(1, \sigma^4(n))-\frac{\theta}{2}\right)\left(f_n(3, \sigma^4(n))-\frac{1-\theta}{2}\right)\\
	=&\left[\lambda+\frac{\theta (1-\lambda)}{2} \right]\left( \frac{\theta}{2}(x_{n;\theta}+z_{n;1-\theta})-\frac{u_{n;\theta}+v_{n;\theta}}{2} \right)+\frac{\theta (1-\lambda)}{2} \left( \frac{\theta }{2}\left( x_{n;\theta}+z_{n;1-\theta}\right)-v_{n;\theta} \right)\\
	&+\frac{(1-\theta) (1-\lambda)}{2}\left( \theta z_{n;1-\theta}+\frac{\theta}{2(1-\theta)}\left[\left(u_{n;\theta}-\frac{\theta}{2}x_{n;\theta}\right)+3\left(v_{n;\theta}-\frac{\theta}{2}x_{n;\theta}\right)\right] \right)\\
	=&\frac{\theta z_{n;1-\theta}}{2}-\frac\lambda2\left(u_{n;\theta}-\frac{\theta}{2}x_{n;\theta}\right)-\frac\lambda2\left(v_{n;\theta}-\frac{\theta}{2}x_{n;\theta}\right).
	\end{split}
	\end{equation*}
	
	\setItemnumber{10}
	\item 
	Lemma \ref{xnun} gives
	$$
	\mathbf{E}\left(f_n(1, \sigma^3(n))-\frac{\theta}{2}\right)\left(f_n(3, \sigma^3(n))-\frac{1-\theta}{2}\right)=\frac{\theta z_{n;1-\theta}}{2}+\frac{\theta}{1-\theta}\left(w_{n;1-\theta}-\frac{1-\theta}{2}x_{n;1-\theta}\right),
	$$	
	and then we have
	$$
	\begin{aligned}[t]
	&\mathbf{E}\left(Y_{3j}(n)-\frac{1-\theta}{2}\right)\left(Y_{4j}(n)-\frac{1-\theta}{2}\right)
	=\frac{1-\theta}{2}y_{n;1-\theta}+\frac\lambda2\left(u_{n;\theta}-\frac{\theta}{2}x_{n;\theta}\right)+\frac{3\lambda}2\left(v_{n;\theta}-\frac{\theta}{2}x_{n;\theta}\right)-\lambda\left(w_{n;1-\theta}-\frac{1-\theta}{2}x_{n;1-\theta}\right).
	\end{aligned}
	$$
\end{enumerate}
\end{proof}
	
\subsection{An equivalent condition for non-reconstruction}
\label{sec3}\noindent 

If the reconstruction problem is solvable, $\sigma(n)$ contains
significant information of the root variable. This can be expressed
in several equivalent ways (\cite{mossel2001reconstruction, mossel2004survey}). 
\begin{lemma}
	\label{equivalent} The non-reconstruction is equivalent to
	$$
	\lim_{n\to \infty}x_{n;\theta}=\lim_{n\to \infty}x_{n;1-\theta}=0.
	$$
\end{lemma}

\section{Recursive formulas}
\label{Sec:Distributional_Recursion}
\subsection{Distributional recursion}
\label{Zresults}
In this section, we will explore the asymptotic behavior of $x_{n;\theta}$
as $n$ goes to infinity, which plays a crucial rule in determining the
reconstructibility. However, it is extremely challenging to
obtain an explicit expression for $x_{n;\theta}$. Therefore, we analyze the recursive relation between $x_{n;\theta}$ and
$x_{n+1;\theta}$ on the tree structure instead. Consider $A$ as a
configuration on $L(n+1)$ and let $A_j$ be its restriction to
$\mathbb{T}_{u_j}\bigcap L(n+1)$. Then from the Markov random field
property, we have
\begin{equation}
\label{recursion} f_{n+1}(1,A)=\frac{N_1}{N_1+N_2+N_3+N_4},
\end{equation}
where
\begin{equation*}
\begin{split}
N_1=&\frac\theta2\prod_{j=1}^d\left[\sum_{i=1}^4P_{1i}\mathbf{P}(\sigma_j(n+1)=A_j\mid\sigma_{u_j}=i)\right]
=\frac\theta2\prod_{j=1}^d\left[1+\frac{2\lambda}{\theta}\left(f_n(1,A_j)-\frac\theta2\right)\right]\mathbf{P}(\sigma_j(n+1)=A_j);
\end{split}
\end{equation*}
\begin{equation*}
\begin{split}
N_2=&\frac\theta2\prod_{j=1}^d\left[\sum_{i=1}^4P_{2i}\mathbf{P}(\sigma_j(n+1)=A_j\mid\sigma_{u_j}=i)\right]=\frac\theta2\prod_{j=1}^d\left[1+\frac{2\lambda}{\theta}\left(f_n(2,A_j)-\frac\theta2\right)\right]\mathbf{P}(\sigma_j(n+1)=A_j);
\end{split}
\end{equation*}
\begin{equation*}
\begin{split}
N_3=&\frac{1-\theta}2\prod_{j=1}^d\left[\sum_{i=1}^4P_{3i}\mathbf{P}(\sigma_j(n+1)=A_j\mid\sigma_{u_j}=i)\right]=\frac{1-\theta}2\prod_{j=1}^d\left[1+\frac{2\lambda}{1-\theta}\left(f_n(3,A_j)-\frac{1-\theta}2\right)\right]\mathbf{P}(\sigma_j(n+1)=A_j);
\end{split}
\end{equation*}
\begin{equation*}
\begin{split}
N_4=&\frac{1-\theta}2\prod_{j=1}^d\left[\sum_{i=1}^4P_{4i}\mathbf{P}(\sigma_j(n+1)=A_j\mid\sigma_{u_j}=i)\right]=\frac{1-\theta}2\prod_{j=1}^d\left[1+\frac{2\lambda}{1-\theta}\left(f_n(4,A_j)-\frac{1-\theta}2\right)\right]\mathbf{P}(\sigma_j(n+1)=A_j).
\end{split}
\end{equation*}
Setting $A=\sigma^1(n+1)$, we have
$$
f_{n+1}(1,\sigma^1(n+1))=\frac{\frac\theta2Z_1}{\frac\theta2Z_1+\frac\theta2Z_2+\frac{1-\theta}2Z_3+\frac{1-\theta}2Z_4},
$$
where
$$
Z_i=\left\{\begin{array}{ll} \prod_{j=1}^d\left[1+\frac{2\lambda}{\theta}\left(Y_{ij}(n)-\frac\theta2\right)\right] \quad & \quad \textup{for}\ i=1, 2,
\\
\
\\
\prod_{j=1}^d\left[1+\frac{2\lambda}{1-\theta}\left(Y_{ij}(n)-\frac{1-\theta}2\right)\right] \quad & \quad \textup{for}\ i=3, 4.
\end{array}
\right.
$$

In the next two lemmas, we provide some important identities
regarding $Z_i(n)$.
\begin{lemma}
	\label{Z1Z2} For any nonnegative $n\in \mathbb{Z}^+$, we have
	$$
	\mathbf{E}\left(Z_1(n)Z_2(n)\right)=\mathbf{E}Z_2^2(n).
	$$
\end{lemma}
\begin{proofsect}{Proof}
	For any configuration $A=(A_1, \ldots, A_d)$ on the $(n+1)$th level, with
	$A_j$ denoting the spins on $L_{n+1}\cap \mathbb{T}_{u_j}$, we have
	\begin{eqnarray*}
		Z_i=\frac{2}{\theta}\frac{\mathbf{P}(\sigma(n+1)=A)}{\prod_{j=1}^d\mathbf{P}(\sigma_j(n+1)=A_j)}\mathbf{P}(\sigma_\rho=i\mid\sigma(n+1)=A), \quad \text{for}\; i=1, 2.
	\end{eqnarray*}
	By the symmetry of the tree, we have
	\begin{equation*}
	\begin{aligned}
	&\mathbf{E}\left(Z_1(n)Z_2(n)\right)\\
	=&\left(\frac2\theta\right)^2\sum_A\left(\frac{\mathbf{P}(\sigma(n+1)=A)}{\prod_{j=1}^d\mathbf{P}(\sigma_j(n+1)=A_j)}\right)^2\mathbf{P}(\sigma_\rho=1\mid\sigma(n+1)=A)\mathbf{P}(\sigma_\rho=2\mid\sigma(n+1)=A)
\mathbf{P}(\sigma(n+1)=A\mid\sigma_\rho=1)
	\\
	=&\left(\frac2\theta\right)^2\sum_A\left(\frac{\mathbf{P}(\sigma(n+1)=A)}{\prod_{j=1}^d\mathbf{P}(\sigma_j(n+1)=A_j)}\right)^2\mathbf{P}^2(\sigma_\rho=1\mid\sigma(n+1)=A)
	\mathbf{P}(\sigma(n+1)=A\mid\sigma_\rho=2)
	\\
	=&\left(\frac2\theta\right)^2\sum_A\left(\frac{\mathbf{P}(\sigma(n+1)=A)}{\prod_{j=1}^d\mathbf{P}(\sigma_j(n+1)=A_j)}\right)^2\mathbf{P}^2(\sigma_\rho=2\mid\sigma(n+1)=A)
	\mathbf{P}(\sigma(n+1)=A\mid\sigma_\rho=1)
	\\
	=&\mathbf{E}Z_2^2.
	\end{aligned}
	\end{equation*}
\end{proofsect}

The means and
variances of monomials of $Z_i$ can be approximated, using the notation $O_\theta$ to emphasize that the constant associated with the $O$-term depends on $\theta$ only, as follows:
\begin{lemma} One has
\begin{enumerate}[(i)]
	\item $
	\!
	\begin{aligned}[t]
	\mathbf{E}Z_1
	=1+ d\lambda^2\frac2\theta x_{n;\theta}+\frac{d(d-1)}{2}\lambda^4\frac4{\theta^2}x_{n;\theta}^2+O_\theta(x_{n;\theta}^3).
	\end{aligned}
	$
	
	\item $
	\!
	\begin{aligned}[t]
	\mathbf{E}Z_2
	=1+ d\lambda^2\frac2\theta y_{n;\theta}+\frac{d(d-1)}{2}\lambda^4\frac4{\theta^2}y_{n;\theta}^2+O_\theta(x_{n;\theta}^3).
	\end{aligned}
	$
	
	\item $
	\!
	\begin{aligned}[t]
	\mathbf{E}Z_i
	=1+ d\lambda^2\frac2{1-\theta} z_{n;1-\theta}+\frac{d(d-1)}{2}\lambda^4\frac4{(1-\theta)^2}z_{n;1-\theta}^2+O_\theta(x_{n;\theta}^3), \quad i=3,4.
	\end{aligned} $
	
	\item 
	$
	\mathbf{E}Z_1^2=1+d\Pi_1+\frac{d(d-1)}{2}\Pi_1^2+O_\theta(x_{n;\theta}^3),
	$
	where
	$$\Pi_1=\mathbf{E}\left[1+\frac{2\lambda}\theta\left(Y_{1j}-\frac{\theta}{2}\right)\right]^2-1
	=\frac{6\lambda^2}\theta x_{n;\theta}+\frac{4\lambda^3}{\theta^2}\left(u_{n;\theta}-\frac{\theta }{2}x_{n;\theta}\right).$$
	
	\item 
	$
	\mathbf{E}Z_2^2=\mathbf{E}Z_1Z_2=1+d\Pi_2+\frac{d(d-1)}{2}\Pi_2^2+O_\theta(x_{n;\theta}^3),
	$
	where
	$$\Pi_2=\mathbf{E}\left[1+\frac{2\lambda}\theta\left(Y_{2j}-\frac{\theta}{2}\right)\right]^2-1
	=\frac{2\lambda^2}\theta x_{n;\theta}+\frac{4\lambda^2}\theta y_{n;\theta}+\frac{4\lambda^3}{\theta^2}\left(v_{n;\theta}-\frac{\theta }{2}x_{n;\theta}\right).$$
	
	\item 
	$
	\mathbf{E}Z_i^2=1+d\Pi_3+\frac{d(d-1)}{2}\Pi_3^2+O_\theta(x_{n;\theta}^3), 
	$
	for $i=3,4$, where
	$$\Pi_3=\mathbf{E}\left[1+\frac{2\lambda}{1-\theta}\left(Y_{3j}-\frac{1-\theta}{2}\right)\right]^2-1
	=\frac{4\lambda^2}{1-\theta}z_{n;1-\theta}+\frac{2\lambda^2}{1-\theta} x_{n;1-\theta}+\frac{4\lambda^3}{(1-\theta)^2}\left(w_{n;1-\theta}-\frac{1-\theta }{2}x_{n;1-\theta}\right).$$
	
	\item $
	\mathbf{E}Z_{1}Z_{i}
	=1+d\Pi_4+\frac{d(d-1)}{2}\Pi_4^2+O_\theta(x_{n;\theta}^3),
	$
	for $i=3,4$, where
	\begin{eqnarray*}
		\Pi_4&=&\mathbf{E}\left[1+\frac{2\lambda}\theta\left(Y_{1j}-\frac{\theta}{2}\right)\right]\left[1+\frac{2\lambda}{1-\theta}\left(Y_{3j}-\frac{1-\theta}{2}\right)\right]-1
		\\
		&=&\frac{2\lambda^2}{\theta}x_{n;\theta}+\frac{4\lambda^2}{1-\theta}z_{n;1-\theta}-\frac{2\lambda^3}{\theta(1-\theta)}\left[\left(u_{n;\theta}-\frac{\theta}{2}x_{n;\theta}\right)+\left(v_{n;\theta}-\frac{\theta}{2}x_{n;\theta}\right)\right].
	\end{eqnarray*}
	
	\item $
	\mathbf{E}Z_2Z_i
	=1+d\Pi_5+\frac{d(d-1)}{2}\Pi_5^2+O_\theta(x_{n;\theta}^3)$,
	for $i=3,4$, where
	\begin{eqnarray*}
		\Pi_5=\mathbf{E}\left[1+\frac{2\lambda}\theta\left(Y_{2j}-\frac{\theta}{2}\right)\right]\left[1+\frac{2\lambda}{1-\theta}\left(Y_{3j}-\frac{1-\theta}{2}\right)\right]-1=
		\frac{2\lambda^2}{\theta}y_{n;\theta}+\frac{4\lambda^2}{1-\theta}z_{n;1-\theta}-\frac{4\lambda^3}{\theta(1-\theta)}\left(v_{n;\theta}-\frac{\theta}{2}x_{n;\theta}\right).
	\end{eqnarray*}
	
	\item $
	\mathbf{E}Z_3Z_4
	=1+d\Pi_6+\frac{d(d-1)}{2}\Pi_6^2+O_\theta(x_{n;\theta}^3),
	$
	where
	\begin{eqnarray*}
		\Pi_6&=&\mathbf{E}\left[1+\frac{2\lambda}{1-\theta}\left(Y_{3j}-\frac{1-\theta}{2}\right)\right]\left[1+\frac{2\lambda}{1-\theta}\left(Y_{4j}-\frac{1-\theta}{2}\right)\right]-1
		\\
		&=&\frac{4\lambda^2}{1-\theta}z_{n;1-\theta}+\frac{2\lambda^2}{1-\theta}y_{n;1-\theta}+\frac{2\lambda^3}{(1-\theta)^2}\left[\left(u_{n;\theta}-\frac{\theta}{2}x_{n;\theta}\right)+3\left(v_{n;\theta}-\frac{\theta}{2}x_{n;\theta}\right)\right]-\frac{4\lambda^3}{(1-\theta)^2}\left(w_{n;1-\theta}-\frac{1-\theta}{2}x_{n;1-\theta}\right).
	\end{eqnarray*}
\end{enumerate}
\end{lemma}
	
\subsection{Main expansions of $x_{n+1;\theta}$ and $z_{n+1;1-\theta}$}
In this section, we investigate the second order
recursive relations associated with $x_{n+1;\theta}$ and $z_{n+1;1-\theta}$, with the assistance of the
following identity
\begin{equation}
\label{identity}
\frac{a}{s+r}=\frac{a}{s}-\frac{ar}{s^2}+\frac{r^2}{s^2}\frac{a}{s+r}.
\end{equation}
Plugging $a=\frac{1-\theta}{2}Z_3$ in equation \eqref{identity}, we have
\begin{equation}
\begin{split}
\label{zexpansion}
z_{n+1;1-\theta}+\frac{1-\theta}{2}=&\mathbf{E}\left[\frac{\frac{1-\theta}{2}Z_3}{\frac\theta2Z_1+\frac\theta2Z_2+\frac{1-\theta}2Z_3+\frac{1-\theta}2Z_4}\right]
\\
=&\mathbf{E}\left[\frac{1-\theta}{2}Z_3\right]-\mathbf{E}\left[\frac{1-\theta}{2}Z_3\left(\frac\theta2Z_1+\frac\theta2Z_2+\frac{1-\theta}2Z_3+\frac{1-\theta}2Z_4-1\right)\right]
\\
&+\mathbf{E}\left[\left(\frac\theta2Z_1+\frac\theta2Z_2+\frac{1-\theta}2Z_3+\frac{1-\theta}2Z_4-1\right)^2\frac{\frac{1-\theta}{2}Z_3}{\frac\theta2Z_1+\frac\theta2Z_2+\frac{1-\theta}2Z_3+\frac{1-\theta}2Z_4}\right].
\end{split}
\end{equation}
Next by the results in Section \ref{Zresults} and taking $\mathcal{Z}_{n;\theta}=-z_{n;1-\theta}$, we have
\begin{eqnarray}
\label{eq:Z_dynamics}
\mathcal{Z}_{n+1;\theta}&=& d\lambda^2\mathcal{Z}_{n;\theta}+\frac{d(d-1)}{2}\lambda^4\vast[\frac{2(1-\theta)^2}{\theta} x_{n;\theta}^2-\frac{4(1-\theta)^2}{\theta}x_{n;\theta}\mathcal{Z}_{n;\theta}+\left(\frac4\theta-\frac{4}{1-\theta}-8\right)\mathcal{Z}_{n;\theta}^2\nonumber
\\
&&\quad\quad\quad\quad\quad\quad\quad\quad\quad\quad\quad+\theta\left(x_{n;1-\theta}^2+y_{n;1-\theta}^2\right)\vast]+R_z+V_z,
\end{eqnarray}
where
\begin{equation}
\label{R_z}
R_z=\mathbf{E}\left(\frac\theta2Z_1+\frac\theta2Z_2+\frac{1-\theta}2Z_3+\frac{1-\theta}2Z_4-1\right)^2\left(\frac{\frac{1-\theta}{2}Z_3}{\frac\theta2Z_1+\frac\theta2Z_2+\frac{1-\theta}2Z_3+\frac{1-\theta}2Z_4}-\frac{1-\theta}2\right)
\end{equation}
and
\begin{equation}
\label{V_z}
|V_z|\leq
C_Vx_{n;\theta}^2\left(\left|\frac{u_{n;\theta}}{x_{n;\theta}}-\frac{\theta}{2}\right|+\left|\frac{v_{n;\theta}}{x_{n;\theta}}-\frac{\theta}{2}\right|+x_{n;\theta}\right)+C_Vx_{n;1-\theta}^2\left(\left|\frac{w_{n;1-\theta}}{x_{n;1-\theta}}-\frac{1-\theta}{2}\right|+x_{n;1-\theta}\right)
\end{equation}
with $C_V$ a constant depending on $\theta$ only.

Similarly, there is the recursive relation for $x_{n+1; \theta}$:
\begin{equation}
\label{xexpansion}
\begin{split}
x_{n+1;\theta}=&\mathbf{E}\frac{\theta}{2}Z_1-\mathbf{E}\frac{\theta}{2}Z_1\left(\frac\theta2Z_1+\frac\theta2Z_2+\frac{1-\theta}2Z_3+\frac{1-\theta}2Z_4-1\right)
\\
&+\mathbf{E}\left(\frac\theta2Z_1+\frac\theta2Z_2+\frac{1-\theta}2Z_3+\frac{1-\theta}2Z_4-1\right)^2\frac{\frac{\theta}{2}Z_1}{\frac\theta2Z_1+\frac\theta2Z_2+\frac{1-\theta}2Z_3+\frac{1-\theta}2Z_4}\\
=&d\lambda^2x_{n;\theta}+\frac{d(d-1)}{2}\lambda^4\vast[\left(-6+\frac{2(1-\theta)}{\theta}
+2\theta\right)x_{n;\theta}^2+\left(-\frac{4\theta}{1-\theta}-16\right)z_{n;1-\theta}^2\\
&\quad\quad\quad\quad\quad\quad\quad\quad\quad\quad+(4\theta-16)x_{n;\theta}z_{n;1-\theta}+\theta\left(x_{n;1-\theta}^2+y_{n;1-\theta}^2\right)\vast]+R_x+V_x,
\end{split}
\end{equation}
where
\begin{equation*}
R_x=\mathbf{E}\left(\frac\theta2Z_1+\frac\theta2Z_2+\frac{1-\theta}2Z_3+\frac{1-\theta}2Z_4-1\right)^2\left(\frac{\frac{\theta}{2}Z_1}{\frac\theta2Z_1+\frac\theta2Z_2+\frac{1-\theta}2Z_3+\frac{1-\theta}2Z_4}-\frac{\theta}2\right)
\end{equation*}
and
$$
|V_x|\leq
C_Vx_{n;\theta}^2\left(\left|\frac{u_{n;\theta}}{x_{n;\theta}}-\frac{\theta}{2}\right|+\left|\frac{v_{n;\theta}}{x_{n;\theta}}-\frac{\theta}{2}\right|+x_{n;\theta}\right)+C_Vx_{n;1-\theta}^2\left(\left|\frac{w_{n;1-\theta}}{x_{n;1-\theta}}-\frac{1-\theta}{2}\right|+x_{n;1-\theta}\right).
$$

\section{Concentration analysis}
\label{Sec:Concentration_Analysis}
In order to study the stability of the dynamical system in equation \eqref{eq:Z_dynamics}, we expect $R_z$ and $V_z$ to be just small perturbations in the order of $o\left(x^2_{n;\theta}+x_{n;1-\theta}^2\right)$. The following lemma ensures that $x_n$ does not drop too fast.
\begin{lemma}
	\label{ndtf} For any $\varrho>0$, there exists a constant
	$\gamma=\gamma(\theta, \varrho)>0$, such that for all $n$ when
	$|\lambda|>\varrho$
	$$
	x_{n+1;\theta}\geq \gamma x_{n;\theta}.
	$$
\end{lemma}
\begin{proof}
	
	Define
	$\Delta_n=\mathbf{E}\max\{X_I(n), X_{II}(n)\}$, where $I=\{1, 2\}$ and $II=\{3, 4\}$. 	Then the following inequality holds 
	$$2\left(x_{n;\theta}+\frac\theta2\right)\leq\mathbf{E}f_n(I,\sigma^I(n))\leq
	\Delta_n,$$ 
	where $f_n(I,\sigma^I(n))$ satisfies
	$$
	f_n(1,\sigma^1(n))+f_n(2,\sigma^2(n))\leq f_n(I,\sigma^I(n)).
	$$ 
	Furthermore, Lemma \ref{lemma1}
	indicates that $x_{n;\theta}=\frac{2}{\theta}\mathbf{E}\left(X_1(n)-\frac\theta2\right)^2$. In the sequel, we consider $\theta\geq\frac12$, and the case of $\theta< \frac12$ can be handled similarly. By the Cauchy-Schwarz inequality and using the fact that $$(X_I-\theta)+(X_{II}-(1-\theta))=0,$$ we have
	\begin{eqnarray}
	\label{MLE} \Delta_n&\leq&\theta+\mathbf{E}\max\left\{X_I(n)-\theta,
	X_{II}(n)-(1-\theta)\right\}\nonumber
	\\
	&=&\theta+\mathbf{E}|X_I(n)-\theta|\nonumber
	\\
	&=&\theta+\mathbf{E}|X_1(n)+X_{2}(n)-\theta|\nonumber
	\\
	&\leq&\theta+2\mathbf{E}\left|X_1(n)-\frac\theta2\right|\nonumber
	\\
	&\leq&\theta+2\left(\mathbf{E}\left(X_1(n)-\frac{\theta}{2}\right)^2\right)^{1/2}\nonumber
	\\
	&=&\theta+\sqrt{2\theta x_{n;\theta}}.
	\end{eqnarray}
	
	For a configuration $A=(A_1,\ldots,A_d)$ on
	$L(n+1)$ with $A_j$ on $\mathbb{T}_{u_j}\bigcap L(n+1)$, define
	\begin{eqnarray*}
		f_{n+1}^*(1, A)&=&\mathbf{P}(\sigma_\rho=1\mid\sigma_1(n+1)=A_1)
		\\
		&=&\frac\theta2\frac{\mathbf{P}(\sigma_1(n+1)=A\mid\sigma_\rho=1)}{\mathbf{P}(\sigma_1(n+1)=A)}
		\\
		&=&\frac\theta2\frac{\sum_{i=1}^4P_{1i}\mathbf{P}(\sigma_1(n+1)=A\mid\sigma_{u_1}=i)}{\mathbf{P}(\sigma_1(n+1)=A)}
		\\
		&=&\frac\theta2\left[\frac{2P_{11}}{\theta}f_n(1,
		A)+\frac{2P_{12}}{\theta}f_n(2, A)+\frac{2P_{13}}{1-\theta}f_n(3, A)+\frac{2P_{14}}{1-\theta}f_n(4, A)\right]
		\\
		&=&\frac\theta2\left[1+\frac{2\lambda}{\theta}\left(f_n(1, A)-\frac\theta2\right)\right],
	\end{eqnarray*}
	and hence
	$$
	\mathbf{E}f_{n+1}^*(1, \sigma_1^1(n+1))=\frac\theta2+\lambda^2x_{n;\theta}.
	$$
	Therefore, it follows from equation \eqref{MLE} that
	\begin{eqnarray*}
		\frac\theta2+\lambda^2x_{n;\theta}\leq\Delta_{n+1}
		\leq\frac\theta2+\sqrt{2\theta x_{n+1;\theta}}\leq\frac\theta2+\sqrt{2x_{n+1;\theta}},
	\end{eqnarray*}
	namely,
	\begin{equation}
	\label{nftf} x_{n+1;\theta}\geq \frac{\lambda^4}2x_{n;\theta}^2\geq
	\frac{\varrho^4}2x_{n;\theta}^2.
	\end{equation}

	Noting that $Z_i\geq0$ and then $0\leq\frac{\frac{\theta}{2}Z_1}{\frac\theta2Z_1+\frac\theta2Z_2+\frac{1-\theta}2Z_3+\frac{1-\theta}2Z_4}\leq 1$, it is concluded
	from equation \eqref{xexpansion} that
	\begin{equation}
	\begin{split}
	\label{eq:xn_ineq}
	x_{n+1;\theta}-d\lambda^2x_{n;\theta}\geq &\mathbf{E}\frac{\theta}{2}Z_1-\mathbf{E}\frac{\theta}{2}Z_1\left(\frac\theta2Z_1+\frac\theta2Z_2+\frac{1-\theta}2Z_3+\frac{1-\theta}2Z_4-1\right)-d\lambda^2x_{n;\theta}\\
	\geq &-C_{\theta}x_{n;\theta}^2.
	\end{split}
	\end{equation}
	Thus there exists a $\delta=\delta(\theta,
	\varrho)=\varrho^2/C_\theta>0$ such that if $x_{n;\theta}<\delta$ then
	$$
	x_{n+1;\theta}\geq (d\lambda^2-\varrho^2)x_{n;\theta}\geq
	(d-1)\varrho^2x_{n;\theta}\geq\varrho^2x_{n;\theta}.
	$$
	Also noting that if $x_{n;\theta}\geq \delta$, equation \eqref{nftf} becomes
	$x_{n+1;\theta}\geq\frac{\varrho^4}2\delta x_{n;\theta}$. 
	
	Finally taking
	$\gamma=\min\{\varrho^2, \varrho^4\delta/2\}$ completes the proof.
\end{proof}

It is known that fixed finite
different vertices far away from the root carry little information of the root, based on which, in-depth
concentration results could be established. 

\begin{lemma}
	\label{effectlittle} For any $\varepsilon>0$ and a positive integer
	$k$, there exists $M=M(\theta, \varepsilon, k)$ such that for any
	collection of vertices $v_1,\ldots, v_k\in L(M)$, when $i=1,2$
	$$
	\sup_{i_1,\ldots,
		i_k\in
		\mathcal{C}}\left|\mathbf{P}(\sigma_\rho=i\mid\sigma_{v_j}=i_j, 1\leq j\leq
	k)-\frac\theta2\right|\leq \varepsilon,
	$$
	while for $i=3,4$
	$$
	\sup_{i_1,\ldots,
		i_k\in
		\mathcal{C}}\left|\mathbf{P}(\sigma_\rho=i\mid\sigma_{v_j}=i_j, 1\leq j\leq
	k)-\frac{1-\theta}2\right|\leq \varepsilon.
	$$
\end{lemma}
\begin{proof}
	Denote the transition probability from state $i$ to state $j$ at distance $s$ by $P_{i, j}^s$. 
	Applying Bayes' theorem, we have
	\begin{eqnarray*}
		P_{1, 1}^{s+1}&=&P_{11}P_{1, 1}^s+P_{12}P_{2, 1}^s+P_{13}P_{3, 1}^s+P_{14}P_{4, 1}^s
		\\
		&=&\left(\lambda+\frac{\theta(1-\lambda)}{2}\right)\left(\frac\theta2+\left(1-\frac\theta2\right)\lambda^s\right)
		+\frac{\theta(1-\lambda)}{2}\left(\frac\theta2-\frac\theta2\lambda^s\right)
		+\frac{(1-\theta)(1-\lambda)}{2}\left(\frac\theta2-\frac\theta2\lambda^s\right)
+\frac{(1-\theta)(1-\lambda)}{2}\left(\frac\theta2-\frac\theta2\lambda^s\right)
		\\
		&=&\frac{\theta}{2}+\left(1-\frac{\theta}{2}\right)\lambda^{s+1}.
	\end{eqnarray*}
	Similarly, we can achieve the following results by induction: 
	\begin{equation*}
	\begin{aligned}[c]
	P_{1, 1}^s&=P_{2, 2}^s=\frac\theta2+\left(1-\frac\theta2\right)\lambda^s,\\
	P_{3, 3}^s&=P_{4, 4}^s=\frac{1-\theta}2+\left(1-\frac{1-\theta}2\right)\lambda^s,\\
	P_{3, 1}^s&=P_{4, 1}^s=P_{3, 2}^s=P_{4, 2}^s=\frac{\theta}2-\frac{\theta}2\lambda^s,
	\end{aligned}
	\qquad \qquad
	\begin{aligned}[c]
	P_{1, 2}^s&=P_{2, 1}^s=\frac\theta2-\frac\theta2\lambda^s,\\
	P_{3, 4}^s&=P_{4, 3}^s=\frac{1-\theta}2-\frac{1-\theta}2\lambda^s,\\
	P_{1, 3}^s&=P_{1, 4}^s=P_{2, 3}^s=P_{2, 4}^s=\frac{1-\theta}2-\frac{1-\theta}2\lambda^s.
	\end{aligned}
	\end{equation*}	
	Consequently, under the condition that $d\lambda^2\leq 1$, when $i=1, 2$ one has
	$$ 
	\frac\theta2-d^{-s/2}\leq P_{\ell, i}^s\leq \frac\theta2+d^{-s/2},
	$$
	and when $j=3, 4$ one has
	$$
	\frac{1-\theta}2-d^{-s/2}\leq P_{\ell, j}^s\leq \frac{1-\theta}2+d^{-s/2}.
	$$
	For fixed $\theta$, $d$ and $k$, define
	$$
	B(s)=\max\left\{\frac{\frac\theta2+d^{-s/2}}{\frac\theta2-d^{-s/2}},
	\frac{\frac{1-\theta}2+d^{-s/2}}{\frac{1-\theta}2-d^{-s/2}}\right\},
	$$
	and let $N=N(\theta, k, \varepsilon)$ be a sufficiently large integer
	such that
	$$
	B^k(N)\leq 1+\varepsilon,
	$$
	which holds for the reason that $d^{-\frac s2}\leq2^{-\frac s2}\to0$ as $s\to\infty$ and then
	$B(s)\to1$ uniformly for all $d$.

	Fix an integer $M$ such that $M>kN$ and choose any
	$v_1,\ldots,v_k\in L(M)$. For $0\leq \ell\leq M$, define $n_\ell$ as the number of vertices in distance $\ell$ from the root with a
	decedent in the set $\{v_1,\ldots,v_k\}$, that is,
	$$
	n_\ell=\{v\in L(\ell): \left|\mathbb{T}_v\cap
	\{v_1,\ldots,v_k\}\right|>0\}.
	$$
	Apparently, $n_0=1$ and $n_M=k$. Also, we can easily see that $n_\ell$ is an increasing integer valued function and there must exist some $\ell$ such that
	$n_\ell=n_{\ell+N}$. Let
	$\{\overline{w}_1,\ldots,\overline{w}_{n_\ell}\}$ be the
	vertices in the set $\{v\in L(\ell): |\mathbb{T}_v\cap
	\{v_1,\ldots,v_k\}|>0\}$, and $\{w_1,\ldots,w_{n_\ell}\}$ be the
	vertices in the set $\{v\in L(\ell+N): |\mathbb{T}_v\cap
	\{v_1,\ldots,v_k\}|>0\}$ such that $w_j$ is the descendent of
	$\overline{w}_j$. By the Markov random field property,
	$\{\sigma_{w_j}\}_{j=1,\cdots, n_\ell}$ are conditionally independent given
	$\sigma_{\overline{w}_j}$. The distribution of $\sigma_{w_j}$ given
	$\sigma_{\overline{w}_j}$ is given by
	$$
	\mathbf{P}(\sigma_{w_j}=i_2\mid\sigma_{\overline{w}_j}=i_1)=P_{i_1,i_2}^N.
	$$
	By Bayes' theorem and the Markov random field property, for any
	$i_1,\ldots,i_{n_\ell}\in \mathcal{C}$, we have
	\begin{eqnarray*}
		\frac{\mathbf{P}(\sigma_\rho=1\mid\sigma_{w_j}=i_j, 1\leq j\leq
			n_\ell)}{\mathbf{P}(\sigma_\rho=2\mid\sigma_{w_j}=i_j, 1\leq j\leq
			n_\ell)}
		&=&\frac{\mathbf{P}(\sigma_{w_j}=i_j, 1\leq j\leq
			n_\ell\mid\sigma_\rho=1)}{\mathbf{P}(\sigma_{w_j}=i_j, 1\leq j\leq
			n_\ell\mid\sigma_\rho=2)}
		\\
		&=&\frac{\sum_{h_1,\ldots,h_{n_\ell}\in
				\mathcal{C}}\mathbf{P}(\forall j\ \sigma_{w_j}=i_j\mid\forall j\
			\sigma_{\overline{w}_j}=h_j)\mathbf{P}(\forall j\
			\sigma_{\overline{w}_j}=h_j\mid\sigma_\rho=1)}{\sum_{h_1,\ldots,h_{n_\ell}\in
				\mathcal{C}}\mathbf{P}(\forall j\ \sigma_{w_j}=i_j\mid\forall j\
			\sigma_{\overline{w}_j}=h_j)\mathbf{P}(\forall j\
			\sigma_{\overline{w}_j}=h_j\mid\sigma_\rho=2)}
		\\
		&=&\frac{\sum_{h_1,\ldots,h_{n_\ell}\in
				\mathcal{C}}\mathbf{P}(\forall j\
			\sigma_{\overline{w}_j}=h_j\mid\sigma_\rho=1)\prod_{j=1}^{n_\ell}P_{h_j,i_j}^N}{\sum_{h_1,\ldots,h_{n_\ell}\in
				\mathcal{C}}\mathbf{P}(\forall j\
			\sigma_{\overline{w}_j}=h_j\mid\sigma_\rho=2)\prod_{j=1}^{n_\ell}P_{h_j,i_j}^N}
		\\
		&\leq&B^{n_\ell}(N)\frac{\sum_{h_1,\ldots,h_{n_\ell}\in
				\mathcal{C}}\mathbf{P}(\forall j\
			\sigma_{\overline{w}_j}=h_j\mid\sigma_\rho=1)}{\sum_{h_1,\ldots,h_{n_\ell}\in
				\mathcal{C}}\mathbf{P}(\forall j\
			\sigma_{\overline{w}_j}=h_j\mid\sigma_\rho=2)}
		\\
		&\leq&B^{k}(N)
		\\
		&\leq&(1+\varepsilon).
	\end{eqnarray*}
	Similarly discussions yield $$\frac{\mathbf{P}(\sigma_\rho=2\mid\sigma_{w_j}=i_j, 1\leq j\leq
		n_\ell)}{\mathbf{P}(\sigma_\rho=1\mid\sigma_{w_j}=i_j, 1\leq j\leq
		n_\ell)}\leq1+\varepsilon,$$
	$$
	\frac{\mathbf{P}(\sigma_\rho=1\mid\sigma_{w_j}=i_j, 1\leq j\leq
		n_\ell)}{\mathbf{P}(\sigma_\rho=m\mid\sigma_{w_j}=i_j, 1\leq j\leq
		n_\ell)}\leq \frac{\theta}{1-\theta}(1+\varepsilon), \quad \quad m=3, 4,$$
	and
	$$ \frac{\mathbf{P}(\sigma_\rho=m\mid\sigma_{w_j}=i_j, 1\leq j\leq
		n_\ell)}{\mathbf{P}(\sigma_\rho=1\mid\sigma_{w_j}=i_j, 1\leq j\leq
		n_\ell)}\leq \frac{1-\theta}{\theta}(1+\varepsilon), \quad \quad m=3, 4.
	$$
	Therefore, we obtain
	$$
	\frac\theta2-\varepsilon\leq\mathbf{P}(\sigma_\rho=1\mid\sigma_{w_j}=i_j,
	1\leq j\leq n_\ell)\leq\frac\theta2+\varepsilon.
	$$
	Finally, since $\sigma_\rho$ is conditionally independent of the
	collection $\{\sigma_{v_1},\ldots,\sigma_{v_k}\}$ given
	$\{\sigma_{w_1},\ldots,\sigma_{w_{n_\ell}}\}$, it is concluded that
	$$\sup_{i_1,\ldots,\
		i_k}\left|\mathbf{P}(\sigma_\rho=1\mid\sigma_{v_j}=i_j, 1\leq j\leq
	k)-\frac\theta2\right|
	\leq\sup_{i_1,\ldots,\
		i_{n_\ell}}\left|\mathbf{P}(\sigma_\rho=1\mid\sigma_{w_j}=i_j, 1\leq
	j\leq n_\ell)-\frac\theta2\right|
	\leq\varepsilon.
	$$
	
	The rest follows similarly.
\end{proof}

Now, we are able to bound the remainders $R_z$ and $V_z$ in equation \eqref{R_z} and \eqref{V_z}, using the preceding concentration results.
\begin{lemma}
	\label{concentration} 
	Assume $|\lambda|>\varrho$ for some $\varrho>0$.
	For any $\varepsilon>0$, there exist $N=N(\theta, \varepsilon)$ and
	$\delta=\delta(\theta, \varepsilon, \varrho)>0$, such that if $n\geq N$
	and $x_{n;\theta}, x_{n;1-\theta}\leq\delta$ then 
	$$|R_z|\leq\varepsilon \left(x_{n;\theta}^2+x_{n;1-\theta}^2\right).$$
\end{lemma}

\begin{proof}
	Fix $k$ an integer such that $k>6$. Choose $M$ such that the conclusions of Lemma \ref{effectlittle} hold with bound $\frac\varepsilon2$. Denote
	$v_1,\ldots,v_{|L(M)|}$ as the vertices in $L(M)$, define $\sigma_v^1(n+1)$ as the spins of vertices in
	$\mathbb{T}_v\bigcap L(n+1)$ conditioned on $\sigma_{\rho}=1$, and for $v\in \{v_1,\ldots,v_{|L(M)|}\}$ 
	let
	$$
	W(v)=f_{n+1-M}(1, \sigma_v^1(n+1)).
	$$
    Then, for $i \in\{1,2,3,4\}$, $W(v)$ would be distributed as
	\begin{equation}
	\label{distribution}
	W(v)\sim f_{n+1-M}(1, \sigma^i(n+1-M)),
	\quad \textup{if}\ \sigma_v^1=i. 
	\end{equation}
	Using the recursive formula in equation \eqref{recursion}, the posterior
	probability of a vertex can be written as a function of the posterior probabilities of its children, so there exists a function $H(W_1,\ldots,W_{|L(M)|})$ such that
	$$
	H(W_1,\ldots,W_{|L(M)|})=f_{n+1}(1,
	\sigma^1(n+1))=\frac{\frac\theta2Z_1}{\frac\theta2Z_1+\frac{\theta}{2}Z_2+\frac{1-\theta}2Z_3+\frac{1-\theta}{2}Z_4},
	$$
	where $W_i=W(v_i)$ for $1\leq i\leq |L(M)|$. We can see that $H$ is a
	continuous function, and when $W_i=\theta/2$ for all $i$
	$$H(W_1,\ldots,W_{|L(M)|})=\theta/2.$$
	Therefore, by Lemma \ref{effectlittle}, if there are at most $k$
	vertices in $L(M)$ such that $W(v)\neq\theta/2$  then
	$$
	\left|H(W_1,\ldots,W_{|L(M)|})-\frac\theta2\right|<\frac\varepsilon2,
	$$
	and thus there exists some $\delta=\delta(\varepsilon)>0$ such that
	if
	$$
	\#\left\{v\in L(M): \left|W(v)-\frac\theta2\right|>\delta\right\}\leq k
	$$
	then
	$$
	\left|H(W_1,\ldots,W_{|L(M)|})-\frac\theta2\right|<\varepsilon.
	$$
	Next by the Chebyshev's inequality and equation \eqref{distribution}, the following holds
	\begin{eqnarray*}
		\mathbf{P}\left(\left|W(v)-\frac\theta2\right|>\delta\right)&\leq&
		\delta^{-2}\sum_{i=1}^{4}\mathbf{E}\left(f_{n+1-M}(1, \sigma^i(n+1-M))-\frac\theta2\right)^2
		\\
		&=&\delta^{-2}(u_{n+1-M; \theta}+v_{n+1-M; \theta}+2w_{n+1-M; \theta})
		\\
		&\leq&\frac{\delta^{-2}}{1-\theta}x_{n+1-M; \theta}.
	\end{eqnarray*}
	For the reason that random variables $\{|W(v)-\theta/2|\}_{v\in \{v_1,\ldots,v_{|L(M)|}\}}$ are
	conditionally independent given $\sigma(M)$, there
	exist suitable constants $C(\theta, \varepsilon, \varrho)$ and
	$N(\theta, \varepsilon)$, such that whenever $n\geq N$,
	\begin{eqnarray*}
		\mathbf{P}\left(\left|\frac{\frac\theta2Z_1}{\frac\theta2Z_1+\frac{\theta}{2}Z_2+\frac{1-\theta}2Z_3+\frac{1-\theta}{2}Z_4}-\frac\theta2\right|>\varepsilon\right)
		&\leq&\mathbf{P}\left(\#\left\{v\in L(M): \left|W(v)-\frac\theta2\right|>\delta\right\}> k\right)
		\\
		&=&\sum_A\mathbf{P}\left(\#\left\{v\in L(M): \left|W(v)-\frac\theta2\right|>\delta\right\}>
		k\mid\sigma(M)=A\right)\mathbf{P}(\sigma(M)=A)
		\\
		&\leq&\sum_A\mathbf{P}\left[\mathbf{Bin}\left(|L(M)|,
		\frac{\delta^{-2}}{1-\theta}x_{n+1-M}\right)>k\right]\mathbf{P}(\sigma(M)=A)
		\\
		&\leq&C'x_{n+1-M; \theta}^6
		\\
		&\leq&Cx_{n;\theta}^6,
	\end{eqnarray*}
	where $\mathbf{Bin}\left(\cdot, \cdot\right)$ denotes the Binomial
	distribution and the last inequality follows from Lemma \ref{ndtf}. Similarly, we can show that, when $j=3, 4$,
	$$
	\mathbf{P}\left(\left|\frac{\frac{1-\theta}2Z_j}{\frac\theta2Z_1+\frac{\theta}{2}Z_2+\frac{1-\theta}2Z_3+\frac{1-\theta}{2}Z_4}-\frac{1-\theta}2\right|>\varepsilon\right)\leq
	Cx_{n;1-\theta}^6.
	$$
	
	For any $\eta>0$, it follows from the Cauchy-Schwarz inequality that
	\begin{eqnarray*}
		|R_z|&=&\left|\mathbf{E}\left(\frac\theta2Z_1+\frac\theta2Z_2+\frac{1-\theta}2Z_3+\frac{1-\theta}2Z_4-1\right)^2\left(\frac{\frac{1-\theta}{2}Z_3}{\frac\theta2Z_1+\frac\theta2Z_2+\frac{1-\theta}2Z_3+\frac{1-\theta}2Z_4}-\frac{1-\theta}2\right)\right|
		\\
		&\leq&\eta
		\mathbf{E}\left(\frac\theta2Z_1+\frac\theta2Z_2+\frac{1-\theta}2Z_3+\frac{1-\theta}2Z_4-1\right)^2
		\\
		&&+\mathbf{E}\left(\frac\theta2Z_1+\frac\theta2Z_2+\frac{1-\theta}2Z_3+\frac{1-\theta}2Z_4-1\right)^2\mathbf{I}\left(\left|\frac{\frac{1-\theta}{2}Z_3}{\frac\theta2Z_1+\frac\theta2Z_2+\frac{1-\theta}2Z_3+\frac{1-\theta}2Z_4}-\frac{1-\theta}2\right|>
		\eta\right)
		\\
		&\leq&\eta
		\mathbf{E}\left(\frac\theta2Z_1+\frac\theta2Z_2+\frac{1-\theta}2Z_3+\frac{1-\theta}2Z_4-1\right)^2
		\\
		&&+\mathbf{P}\left(\left|\frac{\frac{1-\theta}{2}Z_3}{\frac\theta2Z_1+\frac\theta2Z_2+\frac{1-\theta}2Z_3+\frac{1-\theta}2Z_4}-\frac{1-\theta}2\right|>
		\eta\right)^{\frac{1}{2}}\left[\mathbf{E}\left(\frac\theta2Z_1+\frac\theta2Z_2+\frac{1-\theta}2Z_3+\frac{1-\theta}2Z_4-1\right)^4\right]^{\frac{1}{2}}.
	\end{eqnarray*}
	Note that the calculations in Section \ref{Zresults} imply that
	$$
	\mathbf{E}\left(\frac\theta2Z_1+\frac\theta2Z_2+\frac{1-\theta}2Z_3+\frac{1-\theta}2Z_4-1\right)^2\leq C_1\left(x_{n;\theta}^2+x_{n;1-\theta}^2\right)
	$$
	and
	$$ \mathbf{E}\left(\frac\theta2Z_1+\frac\theta2Z_2+\frac{1-\theta}2Z_3+\frac{1-\theta}2Z_4-1\right)^4\leq C_2.
	$$ 
	Thus, there
	exist $C_3=C_3(\theta, \eta, \varrho)$ and $N=N(\theta, \eta)$, such that
	if $n>N$ then
	$$
	\mathbf{P}\left(\left|\frac{\frac{1-\theta}2Z_3}{\frac\theta2Z_1+\frac{\theta}{2}Z_2+\frac{1-\theta}2Z_3+\frac{1-\theta}{2}Z_4}-\frac{1-\theta}2\right|>\varepsilon\right)\leq
	C_3x_{n;1-\theta}^6.
	$$
	Finally, taking $\eta=\frac{\varepsilon}{2C_1}$ and
	$\delta=\frac{\varepsilon}{2\sqrt{C_2C_3}}$, we conclude that if $n\geq N$ and
	$x_{n;\theta}, x_{n;1-\theta}\leq\delta$ then
	\begin{eqnarray*}
		|R_z|\leq\eta C_1x_{n;\theta}^2+C_2C_3x_{n;\theta}^3\leq\varepsilon \left(x_{n;\theta}^2+x_{n;1-\theta}^2\right).
	\end{eqnarray*}
\end{proof}

\begin{lemma}
	\label{concentrationforz} Assume $|\lambda|>\varrho$ for some
	$\varrho>0$. For any $\varepsilon>0$, there exist $N=N(\theta,
	\varepsilon)$ and $\delta=\delta(\theta, \varepsilon, \varrho)$, such
	that if $n\geq N$ and $x_{n;\theta}, x_{n;1-\theta}\leq\delta$ then
	$$
	|V_z|\leq\varepsilon\left(x_{n; \theta}^2+x_{n; 1-\theta}^2\right).
	$$
\end{lemma}
\begin{proof}
	Plugging $$a=\left(Z_1-\left(\frac\theta2Z_1+\frac\theta2Z_2+\frac{1-\theta}2Z_3+\frac{1-\theta}2Z_4\right)\right)^2,\quad r=\left(\frac\theta2Z_1+\frac\theta2Z_2+\frac{1-\theta}2Z_3+\frac{1-\theta}2Z_4\right)^2-1, \quad s=1$$ in equation \eqref{identity}, we have
	\begin{eqnarray}
	\label{w}
	u_{n+1; \theta}&=&\mathbf{E}\left(f_{n+1}(1, \sigma^1(n+1))-\frac{\theta}{2}\right)^2\nonumber
	\\
	&=&\mathbf{E}\left(\frac{\frac{\theta}{2}Z_1}{\frac\theta2Z_1+\frac\theta2Z_2+\frac{1-\theta}2Z_3+\frac{1-\theta}2Z_4}-\frac{\theta}{2}\right)^2\nonumber
	\\
	&=&\frac{\theta^2}{4}\mathbf{E}\left(Z_1-\left(\frac\theta2Z_1+\frac\theta2Z_2+\frac{1-\theta}2Z_3+\frac{1-\theta}2Z_4\right)\right)^2\nonumber
	\\
	&&-\frac{\theta^2}{4}\mathbf{E}\left(Z_1-\left(\frac\theta2Z_1+\frac\theta2Z_2+\frac{1-\theta}2Z_3+\frac{1-\theta}2Z_4\right)\right)^2\left(\left(\frac\theta2Z_1+\frac\theta2Z_2+\frac{1-\theta}2Z_3+\frac{1-\theta}2Z_4\right)^2-1\right)\nonumber
	\\
	&&+\mathbf{E}\left(\left(\frac\theta2Z_1+\frac\theta2Z_2+\frac{1-\theta}2Z_3+\frac{1-\theta}2Z_4\right)^2-1\right)^2\left(\frac{\frac{\theta}{2}Z_1}{\frac\theta2Z_1+\frac\theta2Z_2+\frac{1-\theta}2Z_3+\frac{1-\theta}2Z_4}-\frac{\theta}{2}\right)^2\nonumber
	\\
	&=&\frac\theta2 x_{n+1;\theta}+d\lambda^3\left(u_{n;\theta}-\frac{\theta}{2}x_{n;\theta}\right)+O_\theta(x_{n;\theta}^2).
	\end{eqnarray}
	By equation \eqref{eq:xn_ineq}, there exist
	constants $N_1=N_1(\theta)$ and $\delta_1=\delta_1(\theta)$,
	such that if $n\geq N_1$ and $x_{n;\theta}\leq\delta_1$ then
	\begin{eqnarray*}
		\frac{x_{n;\theta}}{x_{n+1;\theta}}
		\leq\frac{x_{n;\theta}}{\left(1-\frac14\right)d\lambda^2x_{n;\theta}}
		=\frac{4}{3}\frac{1}{d\lambda^2}.
	\end{eqnarray*}
	For fixed $k$, it is known that there exists a $\delta_2=\delta_2(\theta, \varrho,
	k)<\delta_1$, such that if $x_{n;\theta}<\delta_2$ then for any $1\leq\ell\leq k$ one has
	$x_{n+\ell; \theta}<2\delta_2$. Therefore, for any
	positive integer $k$, equation \eqref{w} yields
	\begin{eqnarray*}
		\frac{u_{n+k; \theta}}{x_{n+k; \theta}}-\frac{\theta}{2}&=&d\lambda^3\frac{x_{n+k-1; \theta}}{x_{n+k; \theta}}\left(\frac{u_{n+k-1; \theta}}{x_{n+k-1; \theta}}-\frac{\theta}{2}\right)+O_\theta\left(x_{n+k-1; \theta}\frac{x_{n+k-1; \theta}}{x_{n+k; \theta}}\right)
		\\
		&=&(d\lambda^3)^k\left(\prod_{\ell=1}^k\frac{x_{n+\ell-1; \theta}}{x_{n+\ell; \theta}}\right)\left(\frac{u_{n;\theta}}{x_{n;\theta}}-\frac{\theta}{2}\right)+R,
	\end{eqnarray*}
	where
	\begin{eqnarray*}
		\left|(d\lambda^3)^k\left(\prod_{\ell=1}^k\frac{x_{n+\ell-1; \theta}}{x_{n+\ell; \theta}}\right)\right|
		\leq|d\lambda^3|^k\left(\frac{4}{3d\lambda^2}\right)^k=
		\left(\frac{4}{3}|\lambda|\right)^k
	\end{eqnarray*}
	and
	$$
	|R|\leq
	2C\delta_3\left(\sum_{i=1}^k\left(\frac{4}{3d\lambda^2}\right)^i|d\lambda^3|^{i-1}\right)
	\leq2C\delta_3\frac{1-\left(\frac{4}{3}|\lambda|\right)^k}{1-\frac{4}{3}|\lambda|}\frac{4}{3d\lambda^2}
	$$
	with $C$ denoting the $O_\theta$ constant in equation \eqref{w}. By Lemma \ref{lemma1}, it is easy to obtain
	$0\leq\frac{u_{n;\theta}}{x_{n;\theta}}\leq1$, which implies
	$$\left|\frac{u_{n;\theta}}{x_{n;\theta}}-\frac{\theta}{2}\right|\leq 1.$$ Noticing the fact that $|\lambda|\leq d^{-1/2}\leq
	1/\sqrt{2}$, we achieve $\frac{4}{3}|\lambda|<1$. Taking $k=k(\varepsilon)$ sufficiently large and
	$\delta_3=\delta_3(\theta, k, \varepsilon)=\delta_3(\theta,
	\varepsilon)<\delta_2$ sufficiently small, we have
	$$\left|\frac{u_{n+k; \theta}}{x_{n+k;\theta}}-\frac{\theta}{2}\right|<\varepsilon.$$
	Finally, in view of $|\lambda|>\varrho$ and by Lemma \ref{ndtf}, there exists
	$\gamma=\gamma(\theta, \varrho)$ such that
	$x_{n-k; \theta}\leq \gamma^{-k}x_{n;\theta}$, and then by choosing $N=N(\theta,
	\varepsilon, k)=N(\theta, \varepsilon)>N_1+k$ and
	$\delta=\gamma^k\delta_3$, if $x_{n;\theta}\leq\delta$ and $n\geq N$ one has
	\begin{equation}
	\label{wnxn} \left|\frac{u_{n;\theta}}{x_{n;\theta}}-\frac{\theta}{2}\right|< \varepsilon.
	\end{equation}
	
	Similar discussions yield 
	$$
	\left|\frac{v_{n;\theta}}{x_{n;\theta}}-\frac{\theta}{2}\right|< \varepsilon\quad\text{and}\quad \left|\frac{w_{n;1-\theta}}{x_{n;1-\theta}}-\frac{1-\theta}{2}\right|< \varepsilon
	$$ 
	and thus we complete the proof. 
\end{proof}

\section {Proof of the Main Theorem}
\label{Sec:Proof_of_Main_Theorem}
It follows from Lemma \ref{lemma1} that
$$
x_{n;1-\theta}^2+y_{n;1-\theta}^2\geq\frac12\left(x_{n;1-\theta}+y_{n;1-\theta}\right)^2=\frac{2\theta^2}{(1-\theta)^2}z_{n;1-\theta}^2
$$
and equation \eqref{eq:Z_dynamics} implies
\begin{equation*}
	\mathcal{Z}_{n+1;\theta}
	\geq d\lambda^2\mathcal{Z}_{n; \theta}+\frac{d(d-1)}{2}\lambda^4\left[\frac{2(1-\theta)^2}{\theta} \left(x_{n;\theta}-\mathcal{Z}_{n;\theta}\right)^2+\left(\frac4\theta-\frac{4}{1-\theta}-8-\frac{2(1-\theta)^2}{\theta}+\frac{2\theta^3}{(1-\theta)^2}\right)\mathcal{Z}_{n; \theta}^2\right]
	+R_z+V_z.
\end{equation*}
Solving $
\left(\frac4\theta-\frac{4}{1-\theta}-8-\frac{2(1-\theta)^2}{\theta}+\frac{2\theta^3}{(1-\theta)^2}\right)>0,
$
one has
$$
\frac{3+\sqrt{3}}6<\theta\leq1\quad\textup{or}\quad 0\leq\theta<\frac{3-\sqrt{3}}6.
$$
For fixed $\theta\in (\frac{3+\sqrt{3}}6,1]$, there exists $\frac12<\zeta<1$ such that
$$
\frac4\theta-\frac{4}{1-\theta}-8-\frac{2(1-\theta)^2}{\zeta\theta}+\frac{2\zeta\theta^3}{(1-\theta)^2}>0.
$$
To investigate the non-tightness, it would be convenient to assume that $d\lambda^2\geq\frac12$, say,
$|\lambda|\geq\frac{1}{\sqrt{2d}}$. We take $\varrho=\frac{1}{\sqrt{2d}}$
in Lemma \ref{ndtf} and then get $\gamma=\gamma(\theta, d)$. By
Lemma \ref{concentration} and Lemma \ref{concentrationforz}, there
exist $N=N(\theta, \zeta)$ and $\delta=\delta(\theta, d, \zeta)>0$, such that if $n\geq
N$ and $x_{n;\theta}\leq\delta$ then 
\begin{equation}
\label{6.1}
|R_z+V_z|\leq\frac1{16}\frac{2(1-\theta)^2}{\theta}(1-\zeta)x_{n;\theta}^2+\frac{\theta}{16}(1-\zeta)x_{n;1-\theta}^2.
\end{equation}

To accomplish the proof, it suffices to show that when $d\lambda^2$
is close enough to $1$, at least one of $x_{n;\theta}$ and $x_{n;1-\theta}$ does not converge to $0$. 
We apply reductio ad absurdum, by assuming that
\begin{equation}
\label{assumption}
\lim_{n\to\infty}x_{n;\theta}=\lim_{n\to\infty}x_{n;1-\theta}=0.
\end{equation} 
Thus, there exists $\mathcal{N}=\mathcal{N}(\theta, d, \zeta)>N(\theta, \zeta)$, such that whenever $n>\mathcal{N}$, we have $x_{n;\theta}<\delta$.
Denote 
$$
\Gamma=\min\left\{\zeta\frac{2(1-\theta)^2}{\theta},\ \zeta^2\left(\frac4\theta-\frac{4}{1-\theta}-8-\frac{2(1-\theta)^2}{\zeta\theta}+\frac{2\zeta\theta^3}{(1-\theta)^2}\right)\right\}>0
$$
and note that
$$
\frac{d(d-1)}{2}\lambda^4\geq\left(\frac{d\lambda^2}2\right)^2\geq\frac1{16}.
$$
Consequently, by equations \eqref{eq:Z_dynamics} and \eqref{6.1}, it is concluded that if $n>\mathcal{N}$
\begin{equation}
\begin{split}
\label{zbound}
\mathcal{Z}_{n+1;\theta}\geq& d\lambda^2\mathcal{Z}_{n; \theta}+\frac{d(d-1)}{2}\lambda^4\vast[\zeta\frac{2(1-\theta)^2}{\theta} x_{n;\theta}^2-\frac{4(1-\theta)^2}{\theta}x_{n;\theta}\mathcal{Z}_{n; \theta}+\left(\frac4\theta-\frac{4}{1-\theta}-8\right)\mathcal{Z}_{n; \theta}^2+\zeta\theta\left(x_{n;1-\theta}^2+y_{n;1-\theta}^2\right)\vast]
\\
\geq&d\lambda^2\mathcal{Z}_{n; \theta}+\frac{d(d-1)}{2}\lambda^4\vast[\zeta\frac{2(1-\theta)^2}{\theta} \left(x_{n;\theta}-\frac{\mathcal{Z}_{n; \theta}}\zeta\right)^2+\zeta^2\vast(\frac4\theta-\frac{4}{1-\theta}-8-\frac{2(1-\theta)^2}{\zeta\theta}+\frac{2\zeta\theta^3}{(1-\theta)^2}\vast)\left(\frac{\mathcal{Z}_{n;\theta}}\zeta\right)^2\vast]
\\
\geq&d\lambda^2\mathcal{Z}_{n; \theta}+\frac{d(d-1)}{4}\lambda^4\Gamma\left(x_{n;\theta}-\frac{\mathcal{Z}_{n; \theta}}{\zeta}+\frac{\mathcal{Z}_{n;\theta}}{\zeta}\right)^2
\\
\geq&\mathcal{Z}_{n; \theta}\left[d\lambda^2+\frac{d(d-1)}{4}\lambda^4\Gamma x_{n;\theta}\right].
\end{split}
\end{equation}
From the second inequality of equation \eqref{zbound}, we can also conclude 
\begin{equation}
\label{alternative}
\mathcal{Z}_{n+1;\theta}\geq d\lambda^2\mathcal{Z}_{n; \theta}^2+\left(\frac{d\lambda^2}2\right)^2\zeta\frac{2(1-\theta)^2}{\theta} \left(x_{n;\theta}-\frac{\mathcal{Z}_{n; \theta}}\zeta\right)^2\geq \frac18\left(\frac{\mathcal{Z}_{n;\theta}}\zeta\right)^2+\frac{(1-\theta)^2}{16\theta}\left(x_{n;\theta}-\frac{\mathcal{Z}_{n; \theta}}\zeta\right)^2\geq \xi x_{n; \theta}^2,
\end{equation}
where $\xi$ depends only on $\theta$.

Considering the initial point $x_0=1-\frac{\theta}{2}>0$ and by
Lemma \ref{ndtf}, we have
$ x_{n;\theta}\geq x_0\gamma^n$. Define $\varepsilon=\varepsilon(\theta, d, \zeta)=\xi\left(x_0\gamma^{\mathcal{N}}\right)^2> 0$, choose suitable
$|\lambda|<d^{-\frac{1}{2}}$, and then we have
\begin{equation}
\label{inequality}
d\lambda^2+\frac{d(d-1)}{4}\lambda^4\Gamma\varepsilon>1.
\end{equation}
It is easy to see that equation \eqref{alternative} implies that $\mathcal{Z}_{\mathcal{N}+1}\geq
\xi x_{\mathcal{N}}^2\geq
\varepsilon$. Suppose $\mathcal{Z}_n\geq \varepsilon$ for some
$n> \mathcal{N}$, and then it follows from equations \eqref{zbound}
and \eqref{inequality} that
\begin{eqnarray*}
	x_{n+1;\theta}\geq \mathcal{Z}_{n+1;\theta} \geq
	\mathcal{Z}_{n; \theta}\left[d\lambda^2+\frac{d(d-1)}{4}\lambda^4\Gamma\varepsilon\right]
	> \mathcal{Z}_{n; \theta} \geq\varepsilon.
\end{eqnarray*}
Therefore, by induction we have $x_{n;\theta}\geq \mathcal{Z}_{n; \theta} \geq\varepsilon$ for all $n>\mathcal{N}$,
which contradicts to the assumption imposed in equation \eqref{assumption}. Thus, the proof of Theorem \ref{reconstruction} is completed.

\begin{acknowledgements}
	We would like to thank Joseph Felsenstein and S{\'e}bastien Roch for conversations on this subject. 
\end{acknowledgements}


\bibliographystyle{spbasic}      
\bibliography{\jobname}

\end{document}